\documentclass[preprint,10pt]{elsarticle}

\usepackage{amssymb}
\usepackage{amsmath}
\usepackage[all]{xy}
\usepackage{latexsym}
\usepackage{amsthm}
\usepackage{a4wide}
\usepackage{color}
\usepackage{hyperref}
\usepackage[hyphenbreaks]{breakurl}
\input diagxy

\journal{}

\theoremstyle{plain}
  \newtheorem{thm}{Theorem}[subsection]
  \newtheorem{lem}[thm]{Lemma}
  \newtheorem{prop}[thm]{Proposition}
  \newtheorem{cor}[thm]{Corollary}
\theoremstyle{definition}
  \newtheorem{defn}[thm]{Definition}
  \newtheorem{ques}[thm]{Question}
  \newtheorem{exmp}[thm]{Example}
  \newtheorem{rem}[thm]{Remark}

\DeclareMathOperator{\dom}{dom}
\DeclareMathOperator{\cod}{cod}

\DeclareMathOperator{\ob}{ob}

\makeatletter
\def\ps@pprintTitle{%
 \let\@oddhead\@empty
  \let\@evenhead\@empty
  \def\@oddfoot{\vbox{\hsize=\textwidth\scriptsize
  \copyright 2015. This manuscript version is made available under the CC-BY-NC-ND 4.0 license \url{http://creativecommons.org/licenses/by-nc-nd/4.0/}. The published version is available at \url{http://dx.doi.org/10.1016/j.jpaa.2015.10.005}.\\
  }}%
  \let\@evenfoot\@oddfoot}
\makeatother

\renewcommand{\phi}{\varphi}
\newcommand{\da}{\downarrow}

\newcommand{\ua}{\uparrow}
\newcommand{\ra}{\rightarrow}

\newcommand{\lda}{\swarrow}
\newcommand{\rda}{\searrow}
\newcommand{\Lra}{\Longrightarrow}

\newcommand{\oto}{{\to\hspace*{-3.1ex}{\circ}\hspace*{1.9ex}}}

\newcommand{\bv}{\bigvee}
\newcommand{\bw}{\bigwedge}

\newcommand{\dv}{\dashv}

\newcommand{\nat}{\natural}

\newcommand{\al}{\alpha}
\newcommand{\be}{\beta}

\newcommand{\de}{\delta}
\newcommand{\De}{\Delta}
\newcommand{\ga}{\gamma}

\newcommand{\lam}{\lambda}

\newcommand{\CA}{\mathcal{A}}
\newcommand{\CB}{\mathcal{B}}
\newcommand{\CC}{\mathcal{C}}
\newcommand{\CD}{\mathcal{D}}

\newcommand{\CF}{\mathcal{F}}

\newcommand{\CI}{\mathcal{I}}
\newcommand{\CJ}{\mathcal{J}}

\newcommand{\CM}{\mathcal{M}}

\newcommand{\CP}{\mathcal{P}}
\newcommand{\CQ}{\mathcal{Q}}

\newcommand{\BB}{{\bf B}}
\newcommand{\BD}{{\bf D}}
\newcommand{\Bf}{{\bf f}}
\newcommand{\Bg}{{\bf g}}

\newcommand{\Bu}{{\bf u}}
\newcommand{\Bv}{{\bf v}}
\newcommand{\bQ}{{\bf Q}}

\newcommand{\FD}{\mathfrak{D}}

\newcommand{\sA}{{\sf A}}
\newcommand{\sB}{{\sf B}}

\newcommand{\sY}{{\sf Y}}

\newcommand{\si}{{\sf i}}
\newcommand{\sj}{{\sf j}}
\newcommand{\sk}{{\sf k}}

\newcommand{\bbA}{\mathbb{A}}
\newcommand{\bbB}{\mathbb{B}}
\newcommand{\bbC}{\mathbb{C}}

\newcommand{\Arr}{{\bf Arr}}
\newcommand{\Bond}{{\bf Bond}}
\newcommand{\Cat}{{\bf Cat}}
\newcommand{\CAT}{{\bf CAT}}
\newcommand{\CCat}{{\bf CCat}}
\newcommand{\Chu}{{\bf Chu}}

\newcommand{\ChuCon}{{\bf ChuCon}}

\newcommand{\ChuSpan}{{\bf ChuSpan}}

\newcommand{\Dist}{{\bf Dist}}

\newcommand{\Info}{{\bf Info}}
\newcommand{\Mat}{{\bf Mat}}

\newcommand{\Rel}{{\bf Rel}}
\newcommand{\Set}{{\bf Set}}

\newcommand{\Sq}{{\bf Sq}}

\newcommand{\Sup}{{\bf Sup}}
\newcommand{\QCat}{\CQ\text{-}\Cat}

\newcommand{\QCCat}{\CQ\text{-}\CCat}
\newcommand{\QDist}{\CQ\text{-}\Dist}
\newcommand{\QChu}{\CQ\text{-}\Chu}

\newcommand{\QInfo}{\CQ\text{-}\Info}
\newcommand{\QMat}{\CQ\text{-}\Mat}

\newcommand{\rb}{{\rm b}}
\newcommand{\co}{{\rm co}}

\newcommand{\op}{{\rm op}}

\newcommand{\dR}{R^{\da}}
\newcommand{\uR}{R_{\ua}}

\newcommand{\dS}{S^{\da}}
\newcommand{\uS}{S_{\ua}}

\newcommand{\dU}{U^{\da}}

\newcommand{\uV}{V_{\ua}}

\newcommand{\dphi}{\phi^{\da}}
\newcommand{\uphi}{\phi_{\ua}}
\newcommand{\dphiApBp}{(\phi_{\bbA',\bbB'})^{\da}}
\newcommand{\dphiApB}{(\phi_{\bbA',\bbB})^{\da}}
\newcommand{\dphiABp}{(\phi_{\bbA,\bbB'})^{\da}}
\newcommand{\uphiApBp}{(\phi_{\bbA',\bbB'})_{\ua}}
\newcommand{\uphiApB}{(\phi_{\bbA',\bbB})_{\ua}}
\newcommand{\uphiABp}{(\phi_{\bbA,\bbB'})_{\ua}}

\newcommand{\dpsi}{\psi^{\da}}
\newcommand{\upsi}{\psi_{\ua}}

\newcommand{\hphi}{\widehat{\phi}}
\newcommand{\tphi}{\widetilde{\phi}}

\newcommand{\tzeta}{\widetilde{\zeta}}
\newcommand{\heta}{\widehat{\eta}}

\newcommand{\PA}{\CP\bbA}
\newcommand{\PB}{\CP\bbB}

\newcommand{\PdA}{\CP^{\dag}\bbA}
\newcommand{\PdB}{\CP^{\dag}\bbB}

\newcommand{\sYd}{\sY^{\dag}}

\newcommand{\DQ}{\BD(\CQ)}
\newcommand{\BQ}{{\BB}(\CQ)}
\newcommand{\ArrQ}{\Arr(\CQ)}

\newcommand{\ChuConQ}{{\ChuCon}(\CQ)}

\newcommand{\SqQ}{\Sq(\CQ)}

\newcommand{\QB}{\CQ_{\CB}}
\newcommand{\Fix}{{\sf Fix}}

\numberwithin{equation}{section}

\begin{document}

\begin{frontmatter}



\title{Chu connections and back diagonals between $\mathcal{Q}$-distributors}


\author[Y]{Lili Shen\corref{cor}}
\ead{shenlili@yorku.ca}

\author[S]{Yuanye Tao}
\ead{tyymath@foxmail.com}

\author[S]{Dexue Zhang}
\ead{dxzhang@scu.edu.cn}

\cortext[cor]{Corresponding author.}
\address[Y]{Department of Mathematics and Statistics, York University, Toronto, Ontario M3J 1P3, Canada}
\address[S]{School of Mathematics, Sichuan University, Chengdu 610064, China}

\begin{abstract}
Chu connections and back diagonals are introduced as morphisms for distributors between categories enriched in a small quantaloid $\mathcal{Q}$. These notions, meaningful for closed bicategories, dualize the constructions of arrow categories and the Freyd completion of categories. It is shown that, for a small quantaloid $\mathcal{Q}$, the category of complete $\mathcal{Q}$-categories and left adjoints is a retract of the dual of the category of $\mathcal{Q}$-distributors and Chu connections, and it is dually equivalent to the category of $\mathcal{Q}$-distributors and back diagonals. As an application of Chu connections, a postulation of the intuitive idea of reduction of formal contexts in the theory of formal concept analysis is presented, and a characterization of reducts of formal contexts is obtained.
\end{abstract}


\end{frontmatter}

\tableofcontents


\section{Introduction}

This paper is concerned with the following topics:
\begin{itemize} \setlength{\itemsep}{0pt}
\item constructions for closed bicategories that are dual to that of arrow categories and the Freyd completion;
\item morphisms between $\CQ$-distributors (distributors between categories enriched in a small quantaloid $\CQ$); and
\item a suitable notion of reduction of formal contexts for the theory of formal concept analysis.
\end{itemize}

In the rest of the introduction, we will explain how these seemingly different topics are related to each other.  Roughly speaking, the study of formal contexts is a special case of that of $\CQ$-distributors, which again is a special case of that of closed bicategories.

\subsection{Dualizing the constructions of arrow categories and the Freyd completion}

Given a category $\CC$, one has the \emph{arrow category} $\Arr(\CC)$ \cite{MacLane1998}.  With $\CC$-arrows as objects in $\Arr(\CC)$, a morphism from $f:X_1\to Y_1$ to $g:X_2\to Y_2$ in $\Arr(\CC)$ is a pair $(u:X_1\to X_2,\ v:Y_1\to Y_2)$ of $\CC$-arrows such that the \emph{diagonal} of the square
$$\bfig
\square/>`->`->`>/<700,500>[X_1`X_2`Y_1`Y_2;u`f`g`v]
\morphism(0,500)|r|/-->/<700,-500>[X_1`Y_2;]
\efig$$
makes sense: $gu=vf$. Define a congruence on $\Arr(\CC)$ by claiming that $(u,v)\sim(u',v')$ if the squares
$$\bfig
\square/>`->`->`>/<700,500>[X_1`X_2`Y_1`Y_2;u`f`g`v]
\morphism(0,500)|r|/-->/<700,-500>[X_1`Y_2;]
\square(1500,0)/>`->`->`>/<700,500>[X_1`X_2`Y_1`Y_2;u'`f`g`v']
\morphism(1500,500)|r|/-->/<700,-500>[X_1`Y_2;]
\efig$$
have the same diagonal; that is, $gu=vf=gu'=v'f$. The quotient category $\Arr(\CC)/\!\sim$ is then the Freyd completion of $\CC$ (see Grandis \cite{Grandis2000,Grandis2002}). In a word, the Freyd completion of $\CC$ is the category of \emph{diagonals} in $\CC$.

If $\CC$ is furthermore a closed bicategory, for a square of 1-cells in $\CC$ (not necessarily commutative)
$$\bfig
\square/>`->`->`>/<700,500>[X_1`X_2`Y_1`Y_2;u`f`g`v]
\morphism(700,500)|r|/-->/<-700,-500>[X_2`Y_1;]
\efig$$
one can form the right extension
$$f\lda u:X_2\to Y_1$$
of $f$ along $u$ and the right lifting
$$v\rda g:X_2\to Y_1$$
of $g$ through $v$ \cite{Lack2010}. We say that $(u,v)$ is a \emph{Chu connection} from $f$ to $g$ if the right extension $f\lda u$ is isomorphic to the right lifting $v\rda g$  (i.e., the \emph{back diagonal} of the above square makes sense). 1-cells in $\CC$ and Chu connections between them constitute a category $\ChuCon(\CC)$, called the category of Chu connections in $\CC$.  The term ``Chu connection'' is chosen because: Chu connections in a closed bicategory $\CC$ are a special kind of \emph{Chu spans} in the sense of Koslowski \cite{Koslowski2006} (see Remark \ref{Chu span}); Chu connections between $\CQ$-distributors (will be explained later) extend the notion of Chu transforms, and in particular extend Galois connections between partially ordered sets.

By identifying Chu connections $(u,v),(u',v'):f\to g$ whose corresponding back diagonals
$$\bfig
\square/>`->`->`>/<700,500>[X_1`X_2`Y_1`Y_2;u`f`g`v]
\morphism(700,500)|r|/-->/<-700,-500>[X_2`Y_1;]
\square(1500,0)/>`->`->`>/<700,500>[X_1`X_2`Y_1`Y_2;u'`f`g`v']
\morphism(2200,500)|r|/-->/<-700,-500>[X_2`Y_1;]
\efig$$
are isomorphic, i.e., $f\lda u\cong v\rda g\cong f\lda u'\cong v'\rda g$, one obtains a congruence on $\ChuCon(\CC)$; the resulting quotient category, $\BB(\CC)$, is called the category of \emph{back diagonals} in $\CC$. The construction of back diagonals is clearly dual to that of the Freyd completion: one concerns the diagonals, and the other concerns the back diagonals.

As a first step towards the study of Chu connections, we confine ourselves in this paper to a special kind of closed bicategories: quantaloids. A quantaloid is a locally (partially) ordered and locally complete closed bicategory; or equivalently, a $\Sup$-enriched category with $\Sup$ denoting the symmetric monoidal closed category of complete lattices and join-preserving maps \cite{Rosenthal1996}. Quantaloids may also be thought of as quantales with many objects, in the sense that a unital quantale is a monoid in $\Sup$. For such bicategories, 2-cells are given by (partial) order and isomorphic 1-cells are necessarily identical; so, manipulations of 1-cells in a quantaloid will be much easier than in a general closed bicategory.

\subsection{Morphisms between distributors}

\emph{Distributors} \cite{Benabou2000,Borceux1994a,Borceux1994b} (also known as \emph{profunctors} or \emph{bimodules}) generalize functors in the same way as relations generalize maps. Once we have distributors at hand, it is tempting to ask whether there is a sensible notion of morphisms between them. There are several natural candidates in some special cases. First, adjoint functors can be thought of as morphisms between identity distributors, and they are the prototype of Chu transforms (see below). Second, if $\CC$, $\CD$ are ordinary categories (or, categories enriched over a symmetric monoidal closed category), then distributors from $\CC$ to $\CD$ are functors defined on $\CD^{\op}\times\CC$ \cite{Benabou2000,Borceux1994a} (or, $\CD^{\op}\otimes\CC$ \cite{Borceux1994b});  so, natural transformations can be employed to play the role of morphisms. The limitation of these two approaches is obvious: they make sense only for special kinds of distributors.  The case of morphisms between \emph{any} pair of distributors between categories enriched in a bicategory is much more complicated.    In this paper we present an approach to this problem in a special case, i.e., for distributors  between categories enriched in a  small quantaloid $\CQ$.

Basic notions about quantaloid-enriched categories can be found in \cite{Heymans2010,Heymans2012a,Shen2014,Shen2015,Stubbe2005,Stubbe2006}. Chu transforms have been considered in \cite{Shen2013a} (called infomorphisms there) as morphisms of $\CQ$-distributors for the purpose of studying the functoriality of generalized Dedekind-MacNeille completion. Explicitly, a Chu transform between $\CQ$-distributors $\phi:\bbA\oto\bbB$, $\psi:\bbA'\oto\bbB'$ is a pair of $\CQ$-functors $F:\bbA\to\bbA'$, $G:\bbB'\to\bbB$ with $\psi(F-,-)=\phi(-,G-)$. Chu transforms generalize adjoint functors in the sense that any adjoint pair of $\CQ$-functors $F\dv G:\bbB\to\bbA$ is a Chu transform $(F,G):(\bbA:\bbA\oto\bbA)\to(\bbB:\bbB\oto\bbB)$ between identity $\CQ$-distributors. As Chu transforms between $\CQ$-distributors originate from the theory of Chu spaces developed in \cite{Barr1991,Pratt1995}, it is noteworthy to point out that if one considers a commutative unital quantale $Q$ instead of a general quantaloid $\CQ$, then the category of $Q$-distributors and Chu transforms would exactly be the $*$-autonomous completion $\QCat_{\bot}$ of $\QCat$ with $\bot=Q$ in the sense of Barr \cite{Barr1991}.

Since the category $\QDist$ of $\CQ$-categories (as objects) and $\CQ$-distributors (as arrows) is itself a closed bicategory and, indeed a quantaloid \cite{Rosenthal1996,Stubbe2005}, Chu connections and back diagonals can be constructed in this category. The resulting categories (indeed quantaloids), $\ChuCon(\QDist)$ and $\BB(\QDist)$, have $\CQ$-distributors as objects. So, Chu connections and back diagonals are natural morphisms between $\CQ$-distributors.

Chu connections between $\CQ$-distributors are extensions of Chu transforms: in fact, each Chu transform $(F,G):\phi\to\psi$ between $\CQ$-distributors induces a Chu connection $(F^{\nat},G_{\nat}):\psi\to\phi$, where $F^{\nat}$, $G_{\nat}$ are respectively the cograph and graph of $F$, $G$.

The main results in this paper  (Theorems \ref{M_J_id}, \ref{main} and \ref{main result 3}) are about Chu connections and back diagonals between $\CQ$-distributors. It is shown that, for a small quantaloid $\CQ$, the category $\QCCat$  of complete $\CQ$-categories and left adjoint $\CQ$-functors is a retract of the dual of $\ChuCon(\QDist)$, and $\QCCat$ is dually equivalent to $\BB(\QDist)$. These results also  justify the constructions of Chu connections and back diagonals.

\subsection{Reduction of formal contexts}

If the quantaloid $\CQ$ is the two-element Boolean algebra ${\bf 2}=\{0,1\}$, then a $\CQ$-distributor between discrete $\CQ$-categories degenerates to a relation between sets, hence a \emph{formal context} from the viewpoint of formal concept analysis \cite{Davey2002,Ganter1999}. So, morphisms between formal contexts are a special case of that between $\CQ$-distributors. We point out here  that bonds and Chu correspondences between formal contexts, respectively introduced by Ganter \cite{Ganter2007} and Mori \cite{Mori2008}, are both essentially back diagonals (see Proposition \ref{BQ_Bond_iso} and Subsection \ref{ChuCor}).

In formal concept analysis, every formal context is associated with a complete lattice, called its concept lattice. This process extends the Dedekind-MacNeille completion of partially ordered sets. Since different formal contexts may have  isomorphic concept lattices, reduction of formal contexts is an important problem in formal concept analysis, which aims to reduce the size of input data without changing the structure of the output concept lattice. However, to our knowledge, the intuitive idea of a ``reduct'' of a formal context still lacks a rigorous postulation (see the introductory paragraphs of Section \ref{Reduction}). In this paper, we present a notion of reducts of formal contexts with the help of Chu connections. Indeed, we will develop a general theory in this regard, i.e., a theory of reducts of $\CQ$-distributors for a small quantaloid $\CQ$.

The construction of concept lattices out of a formal context has been extended to $\CQ$-distributors in \cite{Shen2013a}, yielding a complete $\CQ$-category $\CM\phi$ for each $\CQ$-distributor $\phi$.  Given a $\CQ$-distributor $\phi:\bbA\oto\bbB$ and $\CQ$-subcategories $\bbA'\subseteq\bbA$, $\bbB'\subseteq\bbB$, four natural $\CQ$-functors are constructed between $\CM\phi$ and $\CM\phi_{\bbA',\bbB'}$, where $\phi_{\bbA',\bbB'}$ is the restriction of $\phi$ to $\bbA'$ and $\bbB'$, which can be regarded as comparison $\CQ$-functors. A little surprisingly (and fortunately), it is proved that if one of these four $\CQ$-functors is an isomorphism then so are the other three. Based on this fact, the notion of a reduct of a $\CQ$-distributor, in particular, a reduct of a formal context, is postulated. Finally, a characterization of reducts is obtained in terms of reducible $\CQ$-subcategories.


\section{Chu connections and back diagonals in a quantaloid} \label{ChuConQ}

\subsection{Quantaloids}

A \emph{quantaloid} \cite{Rosenthal1996} $\CQ$ is a locally ordered 2-category whose (small) hom-sets are complete lattices such that the composition $\circ$ of arrows preserves joins in each variable. The top and bottom arrows in $\CQ(X,Y)$ are denoted by $\top_{X,Y}$ and $\bot_{X,Y}$, respectively. The corresponding adjoints induced by the compositions
\begin{align}
-\circ f\dv -\lda f:&\ \CQ(X,Z)\to\CQ(Y,Z),\label{circ_dv_lda}\\
g\circ -\dv g\rda -:&\ \CQ(X,Z)\to\CQ(X,Y) \label{circ_dv_rda}
\end{align} satisfy
\begin{equation} \label{composition_implication_adjoint}
g\circ f\leq h\iff g\leq h\lda f\iff f\leq g\rda h
\end{equation}
for all $\CQ$-arrows $f:X\to Y$, $g:Y\to Z$, $h:X\to Z$. These adjoints will be called \emph{left} and \emph{right implications} because of the direction of the arrows in the notations, respectively,  instead of right extensions and right liftings as in the vocabulary of bicategory theory \cite{Lack2010}.

A homomorphism between quantaloids is an ordinary functor between the underlying categories that preserves joins of arrows. A homomorphism of quantaloids is \emph{full} (resp. \emph{faithful}, an \emph{equivalence} of quantaloids, an \emph{isomorphism} of quantaloids) if the underlying functor is full (resp. faithful, an equivalence of underlying categories, an isomorphism of underlying categories).

A pair of $\CQ$-arrows $f:X\to Y$ and $g:Y\to X$ form an \emph{adjunction} $f\dv g$ in $\CQ$ (as a 2-category) if $1_X\leq g\circ f$ and $f\circ g\leq 1_Y$. The following identities are useful for calculations related to adjoint $\CQ$-arrows:

\begin{prop} {\rm\cite{Heymans2010}} \label{adjoint_arrow_calculation}
If $f\dv g$ in a quantaloid $\CQ$, then the following identities hold for all $\CQ$-arrows $h,h'$ such that the operations make sense:
\begin{itemize}
\item[\rm (1)] $h\circ f=h\lda g$, $g\circ h=f\rda h$.
\item[\rm (2)] $(f\circ h)\rda h'=h\rda(g\circ h')$, $(h'\circ f)\lda h=h'\lda(h\circ g)$.
\item[\rm (3)] $(h\rda h')\circ f=h\rda(h'\circ f)$, $g\circ(h'\lda h)=(g\circ h')\lda h$.
\item[\rm (4)] $g\circ(h\rda h')=(h\circ f)\rda h'$, $(h'\lda h)\circ f=h'\lda(g\circ h)$.
\end{itemize}
\end{prop}

A \emph{subquantaloid} of $\CQ$ is exactly a subcategory of $\CQ$ that is closed under the inherited joins of $\CQ$-arrows. A subquantaloid of $\CQ$ is \emph{reflective} (resp. \emph{coreflective}) if it is a reflective (resp. coreflective) subcategory of the underlying category of $\CQ$ such that the corresponding left (resp. right) adjoint of the inclusion functor is a quantaloid homomorphism.

A \emph{congruence} $\vartheta$ on a quantaloid $\CQ$ is a congruence on the underlying category that is compatible with joins of $\CQ$-arrows. In elementary words, a congruence $\vartheta$ consists of a family of equivalence relations $\vartheta_{X,Y}$ on each $\CQ(X,Y)$ $(X,Y\in\ob\CQ)$ such that
\begin{itemize}
\item $(f,f')\in\vartheta_{X,Y}$ and $(g,g')\in\vartheta_{Y,Z}$ implies $(g\circ f,g'\circ f')\in\vartheta_{X,Z}$,
\item $(f_i,f'_i)\in\vartheta_{X,Y}$ $(i\in I)$ implies $\Big(\displaystyle\bv_{i\in I}f_i,\displaystyle\bv_{i\in I}f'_i\Big)\in\vartheta_{X,Y}$.
\end{itemize}

Each congruence $\vartheta$ on $\CQ$ induces a \emph{quotient quantaloid} $\CQ/\vartheta$ equipped with the same objects as $\CQ$. Compositions and joins of arrows in $\CQ/\vartheta$ are clearly well defined, and the obvious quotient functor $\CQ\to\CQ/\vartheta$ is a full quantaloid homomorphism.

Quotient quantaloids may also be defined through \emph{quantaloidal nuclei} \cite{Rosenthal1991,Rosenthal1992a}. A \emph{nucleus} on a quantaloid $\CQ$ is a lax functor $\sj:\CQ\to\CQ$ that is an identity on objects and a closure operator on each hom-set. In elementary words, a nucleus $\sj$ consists of a family of order-preserving maps on each $\CQ(X,Y)$ $(X,Y\in\ob\CQ)$ such that $f\leq\sj f$, $\sj\sj f=\sj f$ and $\sj g\circ\sj f\leq\sj(g\circ f)$ for all $f\in\CQ(X,Y)$, $g\in\CQ(Y,Z)$.

Each nucleus $\sj:\CQ\to\CQ$ induces a \emph{quotient quantaloid} $\CQ_{\sj}$ equipped with the same objects as $\CQ$; arrows in $\CQ_{\sj}$ are the fixed points of $\sj$, i.e., $f\in\CQ_\sj(X,Y)$ if $\sj f=f$ for $f\in\CQ(X,Y)$. The identity arrow in $\CQ_\sj(X,X)$ is $\sj(1_X)$; local joins $\bigsqcup$ and compositions $\circ_{\sj}$ in $\CQ_\sj$ are respectively given by
$$\bigsqcup_{i\in I} f_i=\sj\Big(\bv_{i\in I} f_i\Big),\quad g\circ_\sj f=\sj(g\circ f)$$
for $\{f_i\}_{i\in I}\subseteq\CQ_\sj(X,Y)$, $f\in\CQ_\sj(X,Y)$, $g\in\CQ_\sj(Y,Z)$. In addition, $\sj:\CQ\to\CQ_\sj$ is a full quantaloid homomorphism.

It is not difficult to see that there is no essential difference between the two approaches to quotient quantaloids:

\begin{prop} \label{congruence_nucleus}
Each congruence $\vartheta$ on $\CQ$ induces a nucleus $\sj_{\vartheta}:\CQ\to\CQ$ given by
$$\sj_{\vartheta}f=\bv\{f'\mid (f,f')\in\vartheta_{X,Y}\};$$
that is, $\sj_{\vartheta}f$ is the largest $\CQ$-arrow in the equivalence class of $f$. Conversely, each nucleus $\sj:\CQ\to\CQ$ induces a congruence $\vartheta_{\sj}$ on $\CQ$ with
$$(f,f')\in(\vartheta_{\sj})_{X,Y}\iff\sj f=\sj f'.$$
The two correspondences are mutually inverse, and one has isomorphisms of quotient quantaloids
$$\CQ/\vartheta\cong\CQ_{\sj_{\vartheta}}\quad\text{and}\quad\CQ_{\sj}\cong\CQ/\vartheta_{\sj}.$$
\end{prop}

%
%

\subsection{Chu connections in a quantaloid}

Throughout this paper, $\CQ$ always denotes a quantaloid. Being a closed bicategory, Chu connections in $\CQ$ make sense as follows:

\begin{defn}
For $\CQ$-arrows $f:X_1\to Y_1$, $g:X_2\to Y_2$, a \emph{Chu connection} from $f$ to $g$ is a pair of $\CQ$-arrows $(u:X_1\to X_2,\ v:Y_1\to Y_2)$ such that
$$f\lda u=v\rda g.$$
$$\bfig
\square<1000,500>[X_1`X_2`Y_1`Y_2;u`f`g`v]
\morphism(1000,500)|r|/-->/<-1000,-500>[X_2`Y_1;\ f\lda u=v\rda g]
\efig$$
\end{defn}

Given Chu connections $(u,v):f\to g$, $(u',v'):g\to h$, their composition
$$(u',v')\circ(u,v):=(u'\circ u,v'\circ v):f\to h$$
is also a Chu connection, since
\begin{align*}
f\lda(u'\circ u)&=(f\lda u)\lda u'\\
&=(v\rda g)\lda u'\\
&=v\rda(g\lda u')\\
&=v\rda(v'\rda h)\\
&=(v'\circ v)\rda h;
\end{align*}
and so is the join
$$\bv_{i\in I}(u_i,v_i):=\Big(\bv_{i\in I}u_i,\bv_{i\in I}v_i\Big)$$
of Chu connections $(u_i,v_i):f\to g$ $(i\in I)$. With the identity Chu connection on $f:X\to Y$ given by
$$(1_X,1_Y):f\to f,$$
$\CQ$-arrows and Chu connections constitute a quantaloid $\ChuConQ$ with the componentwise local order inherited from $\CQ$.


\begin{rem} \label{Chu span}
Chu connections are a special kind of \emph{Chu spans} in the sense of Koslowski \cite{Koslowski2006}. A Chu span from $f$ to $g$ is a triple $(u,b,v)$ of $\CQ$-arrows such that $b\circ u\leq f$ and $v\circ b\leq g$.
$$\bfig
\square<700,500>[X_1`X_2`Y_1`Y_2;u`f`g`v]
\morphism(700,500)|m|<-700,-500>[X_2`Y_1;\ b]
\place(200,330)[\twoar(-1,0)] \place(500,170)[\twoar(1,0)]
\efig$$
The composition of Chu spans $(u,b,v):f\to g$, $(u',b',v'):g\to h$ is given by
$$(u',b',v')\circ(u,b,v):=(u'\circ u,\ (b\lda u')\wedge(v\rda b'),\ v'\circ v):f\to h,$$
with $(1_X,\top_{X,Y},1_Y):f\to f$ playing as the identity on each $\CQ$-arrow $f:X\to Y$. The category $\ChuSpan(\CQ)$ of $\CQ$-arrows and Chu spans is in fact a quantaloid with the join of $(u_i,b_i,v_i):f\to g$ $(i\in I)$ given by
$$\bv_{i\in I}(u_i,b_i,v_i)=\Big(\bv_{i\in I}u_i,\bw_{i\in I}b_i,\bv_{i\in I}v_i\Big).$$

It is clear that each Chu connection $(u,v):f\to g$ induces a Chu span
$$(u,f\lda u=v\rda g,v):f\to g,$$
exhibiting $\ChuConQ$ as a subquantaloid of $\ChuSpan(\CQ)$.
\end{rem}

\begin{exmp}[Chu connections in the quantaloid $\Rel$ of sets and relations] \label{ChuConRel}
Given a relation $R\subseteq X\times Y$, write
$$\uR\dv\dR:({\bf 2}^Y)^{\op}\to{\bf 2}^X$$
for the (contravariant) Galois connection given by
$$\uR(A)=\{y\in Y\mid\forall x\in A:\ (x,y)\in R\},\quad\dR(B)=\{x\in X\mid\forall y\in B:\ (x,y)\in R\}.$$
Then a Chu connection from $R\subseteq X\times Y$ to $S\subseteq X'\times Y'$ consists of a pair of relations $U\subseteq X\times X'$ and $V\subseteq Y\times Y'$ such that
$$y\in\uR\dU\{x'\}\iff x'\in\dS\uV\{y\}$$
$$\bfig
\Diamond/<-`->`<-`->/<600,250>[{\bf 2}^X`({\bf 2}^{X'})^{\op}`({\bf 2}^Y)^{\op}`{\bf 2}^{Y'};\dU`\uR`\dS`\uV]
\efig$$
for all $x'\in X'$, $y\in Y$.
\end{exmp}

\begin{exmp}[Chu connections in a free quantaloid] \label{ChuConQB}
Each locally small category $\CB$ naturally induces a \emph{free quantaloid} \cite{Rosenthal1991} $\QB$ with
\begin{itemize}
\item $\ob\QB=\ob\CB$,
\item $\QB(X,Y)=\{\Bf\mid\Bf\subseteq\CB(X,Y)\}$ for all $X,Y\in\ob\CB$,
\item $\Bg\circ\Bf=\{g\circ f\mid f\in\Bf,g\in\Bg\}$ for all $\QB$-arrows $\Bf:X\to Y$, $\Bg:Y\to Z$,
\item ${\bf 1}_X=\{1_X\}$ for all $X\in\ob\CB$.
\end{itemize}
For $\QB$-arrows $\Bf:X_1\to Y_1$, $\Bg:X_2\to Y_2$, a Chu connection $(\Bu,\Bv):\Bf\to\Bg$ is a pair of $\QB$-arrows $\Bu:X_1\to Y_1$, $\Bv:X_2\to Y_2$ such that for any $\CQ$-arrow $h:X_2\to Y_1$,
$$h\in\Bf\lda\Bu \iff h\in\Bv\rda\Bg,$$
or equivalently,
$$ \forall u\in\Bu:\ h\circ u\in\Bf\iff  \forall v\in\Bv:\ v\circ h\in\Bg.$$
\end{exmp}


Each quantaloid $\CQ$ is embedded in $\ChuConQ$ via the faithful (but not full) quantaloid homomorphism
\begin{equation} \label{Q_embed_ChuConQ}
\CQ\to\ChuConQ:\ (f:X\to Y)\mapsto((f,f):\top_{X,X}\to\top_{Y,Y}).
\end{equation}
$$\bfig
\square<800,500>[X`Y`X`Y;f`\top_{X,X}`\top_{Y,Y}`f]
\morphism(800,500)|r|/-->/<-800,-500>[Y`X;\top_{Y,X}]
\efig$$
The embedding is both \emph{weak reflective} and \emph{weak coreflective} in the following sense:

\begin{prop} \label{Q_weak_reflective}
The embedding (\ref{Q_embed_ChuConQ}) is a weak left adjoint of the quantaloid homomorphism
$$\dom:\ChuConQ\to\CQ$$
that sends each Chu connection $(u,v):(f:X_1\to Y_1)\to(g:X_2\to Y_2)$ to $u:X_1\to X_2$, and a weak right adjoint of
$$\cod:\ChuConQ\to\CQ$$
that sends $(u,v)$ to $v:Y_1\to Y_2$.
\end{prop}

\begin{proof}
Recall that for a pair of functors $F:\CC\to\CD$, $G:\CD\to\CC$, $F$ is a \emph{weak left adjoint} of $G$ \cite{Kainen1971} if there is a natural transformation $\al:1_{\CC}\to GF$ (as the unit) such that for all $X\in\ob\CC$, $Y\in\ob\CD$ and $\CC$-morphism $g:X\to GY$, there exists a $\CD$-morphism $f:FX\to Y$ (but not necessarily unique) that makes the triangle
\begin{equation} \label{weak_adjoint}
\bfig
\qtriangle<800,500>[X`GFX`GY;\al_X`g`Gf]
\efig
\end{equation}
commute. Dually, $G$ is a \emph{weak right adjoint} of $F$ if $G^{\op}$ is a weak left adjoint of $F^{\op}$.

Now, in order to see that the embedding (\ref{Q_embed_ChuConQ}) is a weak left adjoint of $\dom:\ChuConQ\to\CQ$, just note that $\{1_X\}_{X\in\ob\CQ}$ is the required unit: for all $X\in\ob\CQ$, $(h:Y\to Z)\in\ob\ChuConQ$ and $\CQ$-arrow $g:X\to Y$,
$$(g,\bot_{X,Z}):\top_{X,X}\to h$$
is a Chu connection that makes the triangle (\ref{weak_adjoint}) commute.

For the next claim about $\cod$, first note that the assignment $(u,v)\mapsto(u,v)^{\op}:=(v^{\op},u^{\op})$
$$\bfig
\square<600,500>[X_1`X_2`Y_1`Y_2;u`f`g`v]
\place(1050,250)[\mapsto]
\square(1500,0)<700,500>[Y_2`Y_1`X_2`X_1;v^{\op}`g^{\op}`f^{\op}`u^{\op}]
\efig$$
gives an isomorphism
$$(-)^{\op}:\ChuConQ\to\ChuCon(\CQ^{\op})^{\op}.$$
Then, with the commutative triangle
\begin{equation} \label{cod_op_dom}
\bfig
\qtriangle<1200,500>[\ChuConQ^{\op}`\ChuCon(\CQ^{\op})`\CQ^{\op};(-)^{\op}`\cod^{\op}`\dom]
\efig
\end{equation}
one obtains that, the dual of the embedding (\ref{Q_embed_ChuConQ}) is weak left adjoint to $\cod^{\op}:\ChuConQ^{\op}\to\CQ^{\op}$, and the conclusion thus follows.
\end{proof}

\begin{prop}
$\dom:\ChuConQ\to\CQ$ preserves existing limits, and $\cod:\ChuConQ\to\CQ$ preserves existing colimits.
\end{prop}

\begin{proof}
We only need to prove that $\dom$ preserves limits.  The claim about $\cod$ follows from the commutative diagram (\ref{cod_op_dom}).

Let $D:\CJ\to\ChuConQ$ be a diagram with $Di=(f_i:X_i\to Y_i)$. If there is a limiting cone
$$\al=(u_i,v_i)_{i\in\ob\CJ}:\De(f:X\to Y)\to D$$
over $D$,  we claim that $\be=(u_i)_{i\in\ob\CJ}:\De X\to\dom\cdot D$ is a limiting cone over $\dom\cdot D$. Indeed, for any cone $\ga=(u'_i)_{i\in\ob\CJ}:\De Z\to\dom\cdot D$,
$$(u'_i,\bot_{Z,Y_i}):(\top_{Z,Z}:Z\to Z)\to(f_i:X_i\to Y_i)$$
$$\bfig
\square<800,500>[Z`X_i`Z`Y_i;u'_i`\top_{Z,Z}`f_i`\bot_{Z,Y_i}]
\morphism(800,500)|r|/-->/<-800,-500>[X_i`Z;\top_{X_i,Z}]
\efig$$
is a Chu connection, and gives rise to a cone $\de=(u'_i,\bot_{Z,Y_i})_{i\in\ob\CJ}:\De\top_{Z,Z}\to D$. Since $\al$ is a limiting cone, there exists a unique Chu connection
$$(u,\bot_{Z,Y}):(\top_{Z,Z}:Z\to Z)\to(f:X\to Y)$$
with $u_i\circ u=u'_i$ and $v_i\circ\bot_{Z,Y}=\bot_{Z,Y_i}$ for all $i\in\ob\CJ$. Thus $u:Z\to X$ is the unique $\CQ$-arrow satisfying $u_i\circ u=u'_i$ for all $i\in\ob\CJ$.
\end{proof}

\subsection{Back diagonals in a quantaloid} \label{BQ}

For Chu connections $(u,v),(u',v'):f\to g$ in a quantaloid $\CQ$, we write $(u,v)\sim(u',v')$ if
$$f\lda u=f\lda u',\quad\text{or equivalently},\quad v\rda g=v'\rda g;$$
that is, if the two squares
$$\bfig
\square[X_1`X_2`Y_1`Y_2;u`f`g`v]
\square(1000,0)/>`->`->`>/[X_1`X_2`Y_1`Y_2;u'`f`g`v']
\morphism(500,500)/-->/<-500,-500>[X_2`Y_1;]
\morphism(1500,500)/-->/<-500,-500>[X_2`Y_1;]
\efig$$
have the same \emph{back diagonal}. ``$\sim$'' is clearly an equivalence relation on $\ChuConQ(f,g)$, and it gives rise to a congruence on $\ChuConQ$:

\begin{prop}
The equivalence relation ``$\sim$'' is compatible with compositions and joins of Chu connections.
\end{prop}

\begin{proof}
Suppose $(u_1,v_1)\sim(u'_1,v'_1):f\to g$, $(u_2,v_2)\sim(u'_2,v'_2):g\to h$, then
\begin{align*}
f\lda(u_2\circ u_1)&=(f\lda u_1)\lda u_2\\
&=(f\lda u'_1)\lda u_2&((u_1,v_1)\sim(u'_1,v'_1))\\
&=(v'_1\rda g)\lda u_2&((u'_1,v'_1)\ \text{is a Chu connection})\\
&=v'_1\rda(g\lda u_2)\\
&=v'_1\rda(g\lda u'_2)&((u_2,v_2)\sim(u'_2,v'_2))\\
&=v'_1\rda(v'_2\rda h)&((u'_2,v'_2)\ \text{is a Chu connection})\\
&=(v'_2\circ v'_1)\rda h.
\end{align*}
Thus $(u_2,v_2)\circ(u_1,v_1)\sim(u'_2,v'_2)\circ(u'_1,v'_1)$. In addition, if $(u_i,v_i)\sim(u'_i,v'_i):f\to g$ $(i\in I)$, then
$$f\lda\bv_{i\in I}u_i=\bw_{i\in I}(f\lda u_i)=\bw_{i\in I}(f\lda u'_i)=f\lda\bv_{i\in I}u'_i,$$
and it follows that $\displaystyle\bv_{i\in I}(u_i,v_i)\sim\displaystyle\bv_{i\in I}(u'_i,v'_i)$.
\end{proof}

We shall denote $\BQ:=\ChuConQ/\!\sim$ for the resulting quotient quantaloid and call it the quantaloid of \emph{back diagonals} in $\CQ$. The nucleus $\si:\ChuConQ\to\ChuConQ$ corresponding to the congruence ``$\sim$'' (see Proposition \ref{congruence_nucleus}) is given by
$$\si(u,v):=((f\lda u)\rda f,\ g\lda(v\rda g))$$
for all Chu connections $(u,v):f\to g$. Thus we write
$$\si:\ChuConQ\to\BQ$$
for the quotient homomorphism.

A Chu connection $(u,v)$ is said to be \emph{closed} if $(u,v)=\si(u,v)$, in which case $(u,v)$ is the largest member in its equivalence class and it can be determined solely by $u$ or $v$:

\begin{prop} \label{siuv}
Let $(u,v):f\to g$ be a Chu connection. Then
$$\si(u,v)=((f\lda u)\rda f,\ g\lda(f\lda u))=((v\rda g)\rda f,\ g\lda(v\rda g)).$$
\end{prop}

\begin{exmp}[Continuation of Example \ref{ChuConRel}] \label{BRel}
A Chu connection $(U,V):R\to S$ in the quantaloid $\Rel$  is closed if
$$\dR\uR\dU\{x'\}=\dU\{x'\}\quad\text{and}\quad\uS\dS\uV\{y\}=\uV\{y\}$$
for all $x'\in X'$, $y\in Y$.
\end{exmp}

\begin{exmp}[Continuation of Example \ref{ChuConQB}] \label{BQB}
A Chu connection $(\Bu,\Bv):\Bf\to\Bg$ in the free quantaloid $\QB$ is closed if
$$u\in\Bu\iff\forall h\in\Bf\lda\Bu:\ h\circ u\in\Bf\quad\text{and}\quad v\in\Bv\iff\forall h\in\Bv\rda\Bg:\ v\circ h\in\Bg.$$
\end{exmp}

One may expect that a \emph{back diagonal} from $f:X_1\to Y_1$ to $g:X_2\to Y_2$ should intuitively be a $\CQ$-arrow from $X_2$ to $Y_1$. In fact, there is a description of $\BQ$ reflecting to this intuition, which is particularly useful in Subsection \ref{BQ_quantale}.

A $\CQ$-arrow $b:X_2\to Y_1$ with
$$f\lda(b\rda f)=b=(g\lda b)\rda g$$
$$\bfig
\square<700,500>[X_1`X_2`Y_1`Y_2;b\rda f`f`g`g\lda b]
\morphism(700,500)/-->/<-700,-500>[X_2`Y_1;b]
\efig$$
is called a \emph{bond} from $f$ to $g$ (in generalization of the terminology used for the case $\CQ=\Rel$; see \cite{Ganter2007}). The composition of bonds $b:f\to g$, $b':g\to h$ is given by
\begin{align*}
b'\bullet b&=(g\lda b)\rda b'=((h\lda b')\circ(g\lda b))\rda h\\
&=b\lda(b'\rda g)=f\lda((b'\rda g)\circ(b\rda f)).
\end{align*}
$$\bfig
\square<700,500>[X_2`X_3`Y_2`Y_3;b'\rda g`g`h`h\lda b']
\morphism(700,500)/-->/<-700,-500>[X_3`Y_2;b']
\square(-700,0)<700,500>[X_1`X_2`Y_1`Y_2;b\rda f`f``g\lda b]
\morphism(0,500)/-->/<-700,-500>[X_2`Y_1;b]
\square(1300,0)/->`->``->/<700,500>[X_1`X_2`Y_1`Y_2;b\rda f`f``g\lda b]
\square(2000,0)/->``->`->/<700,500>[X_2`X_3`Y_2`Y_3;b'\rda g``h`h\lda b']
\morphism(2700,500)/-->/<-1400,-500>[X_3`Y_1;b'\bullet b]
\morphism(2700,500)|b|/-->/<-700,-500>[X_3`Y_2;b']
\morphism(2000,500)/-->/<-700,-500>[X_2`Y_1;b]
\place(1000,250)[\mapsto]
\efig$$
$\CQ$-arrows and bonds constitute a quantaloid $\Bond(\CQ)$ with the local order given by the \emph{reversed} order of $\CQ$-arrows. 

\begin{prop} \label{BQ_Bond_iso}
$\BQ$ and $\Bond(\CQ)$ are isomorphic quantaloids.
\end{prop}

\begin{proof}
The assignments $(u,v)\mapsto f\lda u=v\rda g$ and $b\mapsto(b\rda f,g\lda b)$ establish an isomorphism of complete lattices
$$\BQ(f,g)\cong\Bond(\CQ)(f,g),$$
and thus gives rise to the desired isomorphism.
\end{proof}

\subsection{Examples: Back diagonals in an integral quantale} \label{BQ_quantale}

An \emph{integral quantale} $(Q,\&)$ is a one-object quantaloid in which the unit $1$ of the underlying monoid $(Q,\&)$ is the top element of the complete lattice $Q$. In this subsection we describe the quantaloid $\BB(Q)$ of back diagonals in an integral quantale $(Q,\&)$. Of particular interest is the case that the quantale is given by the unit interval $[0,1]$ coupled with a continuous t-norm \cite{Klement2000}. Note that we identify $\BB(Q)$ with $\Bond(Q)$ by Proposition \ref{BQ_Bond_iso}.

For each integral quantale $\CQ=(Q,\&)$,  the quantaloid $\BB(\CQ)$ of back diagonals in $(Q,\&)$ consists of the following data:
\begin{itemize}
\item objects: elements $x,y,z,\dots$ in $Q$;
\item arrows: $\BB(\CQ)(x,y)=\{b\in Q\mid x\lda(b\rda x)=b=(y\lda b)\rda y\}$;
\item composition: for all $b\in\BB(\CQ)(x,y)$, $b'\in\BB(\CQ)(y,z)$,
$$b'\bullet b=(y\lda b)\rda b'=b\lda(b'\rda y);$$
\item the unit in $\BB(\CQ)(x,x)$ is $x$.
\end{itemize}

\begin{prop} \label{BQxy}
For an integral quantale $(Q,\&)$ and $x,y\in Q$, $1\in\BB(\CQ)(x,y)$ and $\BB(\CQ)(x,y)$ is a subset of the upper set generated by $x\vee y$, i.e.,
$\BB(\CQ)(x,y)\subseteq\ua\!(x\vee y).$
\end{prop}

\begin{proof}
Just note that for all $b\in\BB(\CQ)(x,y)$, $x\&(b\rda x)\leq x\& 1=x$ implies $x\leq x\lda(b\rda x)=b$. Similarly $y\leq b$.
\end{proof}

In the case that $Q$ is a commutative quantale, we shall write $x\ra y$ for $y\lda x=x\rda y$.

\begin{exmp}
Let $\CQ=(Q,\&)$ be an MV-algebra \cite{Chang1958}. Since for all $x,b\in Q$, $x\leq b$ implies
$$(b\ra x)\ra x= b\vee x=b,$$
one gets $\BB(\CQ)(x,y)=\ua\!(x\vee y)$ for all $x,y\in Q$. For $b\in\BB(\CQ)(x,y)$, $b'\in\BB(\CQ)(y,z)$, the composition is given by
$$b'\bullet b=(b\ra y)\ra b'=(b'\ra y)\ra b.$$
\end{exmp}

\begin{exmp}
Let $(Q,\&)$ be Lawvere's quantale $([0,\infty]^{\op},+)$, in which $x\ra y=\max\{0,y-x\}$ for all $x,y\in Q$ \cite{Lawvere1973}. Then
$$\BB(Q)(x,y)=\begin{cases}
[0,\min\{x,y\}], & \text{if}\ 0\leq x,y<\infty,\\
\{0,\infty\}, & \text{if}\ x=y=\infty,\\
\{0\}, & \text{otherwise}.
\end{cases}$$
For $b\in\BB(Q)(x,y)$, $b'\in\BB(Q)(y,z)$,
$$b'\bullet b=\begin{cases}
\infty, & \text{if}\ b=b'=y=\infty,\\
0, & \text{if}\ \min\{b,b'\}<\infty\ \text{and}\ y=\infty,\\
\max\{0,b+b'-y\}, & \text{otherwise}.
\end{cases}$$
\end{exmp}

\begin{exmp} \label{t_norm_BQ}
It is well known \cite{Faucett1955,Klement2000,Mostert1957} that a continuous t-norm $\&$ on the unit interval $[0,1]$ can be written as  an ordinal sum of three basic t-norms: the minimum, the product, and the {\L}ukasiewicz t-norm. In this example, we describe the quantaloids of back diagonals with respect to these basic t-norms.
\begin{itemize}
\item (Minimum t-norm) For all $x,y\in[0,1]$, $x\&y =\min\{x,y\}$ and
  $$x\ra y=\begin{cases}
  1, & \text{if}\ x\leq y,\\
  y, & \text{if}\ x>y.
  \end{cases}$$
  Thus
  $$\BB(Q)(x,y)=\begin{cases}
  \{x,1\}, & \text{if}\ 0\leq x=y<1,\\
  \{1\}, & \text{otherwise}.
  \end{cases}$$
  For $b\in\BB(Q)(x,y)$, $b'\in\BB(Q)(y,z)$,
  $$b'\bullet b=\begin{cases}
  b, & \text{if}\ 0\leq b=b'=y<1,\\
  1, & \text{if}\ b=1\ \text{or}\ b'=1.
  \end{cases}$$
\item (Product t-norm) For all $x,y\in[0,1]$, $x\&y=xy$ and
  $$x\ra y=\min\Big\{1,\dfrac{y}{x}\Big\}.$$
  It is easy to see that, as quantales, $[0,1]$ together with the product t-norm is isomorphic to  Lawvere's quantale $([0,\infty]^{\op},+)$. In this case,
  $$\BB(Q)(x,y)=\begin{cases}
  [\max\{x,y\},1], & \text{if}\ 0<x,y\leq 1\\
  \{0,1\}, & \text{if}\ x=y=0,\\
  \{1\}, & \text{otherwise}.
  \end{cases}$$
  For $b\in\BB(Q)(x,y)$, $b'\in\BB(Q)(y,z)$,
  $$b'\bullet b=\begin{cases}
  0, & \text{if}\ b=b'=y=0,\\
  1, & \text{if}\ \max\{b,b'\}>0\ \text{and}\ y=0,\\
  \min\Big\{1,\dfrac{bb'}{y}\Big\}, & \text{otherwise}.
  \end{cases}$$
\item ({\L}ukasiewicz t-norm) For all $x,y\in[0,1]$, $x\&y=\max\{0,x+y-1\}$ and
  $$x\ra y=\min\{1,1-x+y\}.$$
  Thus
  $$\BB(Q)(x,y)=[\max\{x,y\},1].$$
  For $b\in\BB(Q)(x,y)$, $b'\in\BB(Q)(y,z)$,
  $$b'\bullet b=\min\{1,b+b'-y\}.$$
\end{itemize}
\end{exmp}

\subsection{A digression: The categories of arrows and diagonals in a quantaloid}


This subsection, meant to help understand the difference between diagonals and back diagonals in a quantaloid, recalls some basic properties of the arrow category and the Freyd completion of a quantaloid.  The reader is referred to \cite{Stubbe2014} for more on the category of diagonals in a quantaloid.

Given a quantaloid $\CQ$, the arrow category $\ArrQ$ of $\CQ$ (denoted by $\SqQ$ in \cite{Stubbe2014}) has $\CQ$-arrows as objects, and pairs of $\CQ$-arrows ($u:X_1\to X_2,\ v:Y_1\to Y_2$) satisfying
$$g\circ u=v\circ f$$
as arrows from $f:X_1\to Y_1$ to $g:X_2\to Y_2$. $\ArrQ$ is in fact a quantaloid with componentwise local order inherited from $\CQ$.

There is a fully faithful quantaloid homomorphism
\begin{equation} \label{Q_embed_SqQ}
\CQ\to\ArrQ:\ (f:X\to Y)\mapsto((f,f):1_X\to 1_Y)
\end{equation}
$$\bfig
\square<800,500>[X`Y`X`Y;f`1_X`1_Y`f]
\morphism(0,500)|r|/-->/<800,-500>[X`Y;f]
\efig$$
that embeds $\CQ$ in $\ArrQ$ as a both reflective and coreflective subquantaloid:

\begin{prop}
The fully faithful embedding (\ref{Q_embed_SqQ}) is left adjoint to the quantaloid homomorphism
$$\dom:\ArrQ\to\CQ$$
which maps an arrow $(u,v):(f:X_1\to Y_1)\to(g:X_2\to Y_2)$ in $\ArrQ$  to $u:X_1\to X_2$, and right adjoint to
$$\cod:\ArrQ\to\CQ$$
which maps $(u,v)$ to $v:Y_1\to Y_2$.
\end{prop}

\begin{proof}
For all $X\in\ob\CQ$ and $\CQ$-arrows $g:Y\to Z$, the assignment $h\mapsto g\circ h$ gives rise to
$$\CQ(X,Y)\cong\ArrQ(1_X,g),$$
and the assignment $h\mapsto h\circ g$ gives rise to
$$\CQ(Z,X)\cong\ArrQ(g,1_X).$$
\end{proof}

For  arrows $(u,v),(u',v'):f\to g$ in $\ArrQ$, denote by $(u,v)\sim(u',v')$ if the commutative squares
$$\bfig
\square/>`->`->`>/[X_1`X_2`Y_1`Y_2;u`f`g`v]
\square(1000,0)/.>`->`->`.>/[X_1`X_2`Y_1`Y_2;u'`f`g`v']
\morphism(0,500)/-->/<500,-500>[X_1`Y_2;]
\morphism(1000,500)/-->/<500,-500>[X_1`Y_2;]
\efig$$
have the same \emph{diagonal}. ``$\sim$'' gives rise to a congruence on $\ArrQ$, and the induced quotient quantaloid, denoted by $\DQ$, is called the quantaloid of \emph{diagonals} in $\CQ$ \cite{Stubbe2014}. The associated nucleus $\sk:\ArrQ\to\ArrQ$ sends each  arrow $(u,v):f\to g$ in $\ArrQ$ to
$$\sk(u,v):=(g\rda(g\circ u),\ (v\circ f)\lda f),$$  which is the largest  in the equivalence class of $(u,v)$. An arrow $(u,v)$ in $\ArrQ$ is said to be \emph{closed} if $(u,v)=\sk(u,v)$. Analogous to Proposition \ref{siuv} one has:

\begin{prop} \label{kiuv}
For each  arrow $(u,v):f\to g$ in $\ArrQ$,
$$\sk(u,v)=(g\rda(g\circ u),\ (g\circ u)\lda f)=(g\rda(v\circ f),\ (v\circ f)\lda f).$$
\end{prop}


\begin{exmp} \label{SqRel}
The structure of the quantaloid $\Arr(\Rel)$ is easy: an arrow from $R\subseteq X\times Y$ to $S\subseteq X'\times Y'$ consists of a pair of relations $U\subseteq X\times X'$ and $V\subseteq Y\times Y'$ such that $S\circ U=V\circ R$; that is,
$$\exists x'\in X':\ (x,x')\in U\ \text{and}\ (x',y')\in S\iff\exists y\in Y:\ (x,y)\in R\ \text{and}\ (y,y')\in V$$
for all $x\in X$, $y'\in Y'$.

In order to describe the quantaloid $\BD(\Rel)$, it suffices to describe the closed arrows in $\Arr(\Rel)$. For this, note that each relation $R\subseteq X\times Y$ induces a (covariant) Galois connection
$$R^*\dv R_*:{\bf 2}^X\to{\bf 2}^Y$$
with
$$R^*(B)=\{x\in X\mid\exists y\in B:\ (x,y)\in R\},\quad R_*(A)=\{y\in Y\mid\forall x\in X:\ (x,y)\in R\ \text{implies}\ x\in A\}.$$
Then an arrow $(U,V):R\to S$ in $\Arr(\Rel)$ is closed if and only if
$$(S^{\op})_*(S^{\op})^*(U^{\op})^*\{x\}=(U^{\op})^*\{x\}\quad\text{and}\quad R_* R^* V^*\{y'\}=V^*\{y'\}$$
for all $x\in X$, $y'\in Y'$.
\end{exmp}

\begin{exmp} \label{SqQB}
In the free quantaloid $\QB$ generated by a locally small category $\CB$ (see Example \ref{ChuConQB}), $(\Bu,\Bv):\Bf\to\Bg$ is an arrow in $\Arr(\QB)$ if $\Bg\circ\Bu=\Bv\circ\Bf$; that is, for all $g\in\Bg$, $u\in\Bu$, $g\circ u$ factors through some $v\in\Bv$ via some $f\in\Bf$, and vice versa. It is moreover closed if
$$u\in\Bu\iff\forall g\in\Bg:\ g\circ u\in\Bg\circ\Bu\quad\text{and}\quad v\in\Bv\iff\forall f\in\Bf:\ v\circ f\in\Bv\circ\Bf.$$
\end{exmp}

Similar to Proposition \ref{BQ_Bond_iso}, the quantaloid $\DQ$ may be equivalently described by a quantaloid with $\CQ$-arrows as objects, and $\CQ$-arrows $d:X_1\to Y_2$ with
$$(d\lda f)\circ f=d=g\circ(g\rda d)$$
as arrows from $f:X_1\to Y_1$ to $g:X_2\to Y_2$ \cite{Stubbe2014}. This characterization is particularly useful to describe diagonals in an integral quantale $(Q,\&)$ (considered as a one-object quantaloid as in Subsection \ref{BQ_quantale}):

\begin{exmp} \cite{Hohle2011}
If $(Q,\&)$ is a frame, or the quantale $([0,\infty]^{\op},+)$, or the interval $[0,1]$ coupled with a continuous t-norm, then for all $x,y\in Q$,
$$\BD(Q)(x,y)=\da\!(x\wedge y)=\{d\in Q\mid d\leq x\wedge y\}.$$
\end{exmp}

\subsection{Constructing Girard quantaloids}

From a symmetric monoidal closed category one may generate $*$-autonomous categories through the well known \emph{Chu construction} \cite{Barr1991,Pratt1995}. More generally, based on a closed bicategory Koslowski \cite{Koslowski2001} constructed cyclic $*$-autonomous bicategories, extending Barr's work \cite{Barr1995} on the nonsymmetric version of $*$-autonomous categories to a higher order.

\emph{Girard quantaloids} are locally ordered examples of cyclic $*$-autonomous bicategories in the sense of Koslowski \cite{Koslowski2001}. In this subsection, we show that diagonals and back diagonals in a quantaloid can be combined to construct a Girard quantaloid.

In a quantaloid $\CQ$, a family of $\CQ$-arrows $\FD=\{d_X:X\to X\}_{X\in\ob\CQ}$ is called a \emph{cyclic family} (resp. \emph{dualizing family}) if
$$d_X\lda f=f\rda d_Y\quad (\text{resp.}\ (d_X\lda f)\rda d_X=f=d_Y\lda(f\rda d_Y))$$
for all $\CQ$-arrows $f:X\to Y$. A \emph{Girard quantaloid} \cite{Rosenthal1992} is a quantaloid $\CQ$ equipped with a cyclic dualizing family $\FD$ of $\CQ$-arrows.

A one-object Girard quantaloid $\CQ$ is a Girard unital quantale \cite{Yetter1990}, which is an ordered example of cyclic $*$-autonomous categories \cite{Barr1995,Rosenthal1994}; as one expects, a commutative Girard unital quantale is exactly a $*$-autonomous category in the classical sense.

Now let us look at the embeddings (\ref{Q_embed_ChuConQ}) and (\ref{Q_embed_SqQ}) that respectively embed $\CQ$ in $\ChuConQ$ and $\ArrQ$. For each $\CQ$-arrow $f:X\to Y$, the diagonal of the embedding (\ref{Q_embed_SqQ}) is $f:X\to Y$, and the back diagonal of the embedding (\ref{Q_embed_ChuConQ}) is $\top_{Y,X}:Y\to X$; by putting them together one actually gets an assignment $f\mapsto(f,\top_{Y,X})$ that embeds $\CQ$ in a Girard quantaloid:

\begin{prop} \label{Q_QG}
Each quantaloid $\CQ$ is embedded in a Girard quantaloid $\CQ_G$.
\end{prop}

\begin{proof}
Define a quantaloid $\CQ_G$ with $\ob\CQ_G=\ob\CQ$ as follows:
\begin{itemize}
\item for $X,Y\in\ob\CQ$, $\CQ_G(X,Y)=\CQ(X,Y)\times\CQ(Y,X)$;
\item for $\{(f_i,f'_i)\}_{i\in I}\subseteq\CQ_G(X,Y)$, $\displaystyle\bv_{i\in I}(f_i,f'_i)=\Big(\displaystyle\bv_{i\in I}f_i,\displaystyle\bw_{i\in I}f'_i\Big)$;
\item for $\CQ_G$-arrows $(f,f'):X\to Y$, $(g,g'):Y\to Z$, $(h,h'):X\to Z$,
\begin{align*}
(g,g')\circ(f,f')&=(g\circ f,\ (f'\lda g)\wedge(f\rda g')),\\
(h,h')\lda(f,f')&=((h\lda f)\wedge(h'\rda f'),\ f\circ h'),\\
(g,g')\rda(h,h')&=((g\rda h)\wedge(g'\lda h'),\ h'\circ g);
\end{align*}
\item the identity $\CQ_G$-arrow on $X\in\ob\CQ$ is $(1_X,\top_{X,X}):X\to X$.
\end{itemize}

$\CQ_G$ is a quantaloid since
\begin{align*}
(g,g')\circ(f,f')\leq(h,h')&\iff g\circ f\leq h\ \text{and}\ h'\leq f'\lda g\ \text{and}\ h'\leq f\rda g'\\
&\iff g\leq h\lda f\ \text{and}\ g\leq h'\rda f'\ \text{and}\ f\circ h'\leq g'\\
&\iff(g,g')\leq(h,h')\lda(f,f')
\end{align*}
and $(g,g')\circ(f,f')\leq(h,h')\iff(f,f')\leq(g,g')\rda(h,h')$ can be deduced similarly.

Let
$$\FD=\{(\top_{X,X},1_X):X\to X\}_{X\in\ob\CQ},$$
then $\CD$ is a cyclic family since
$$(\top_{X,X},1_X)\lda(f,f')=(f',f)=(f,f')\rda(\top_{Y,Y},1_Y);$$
$\CD$ is a dualizing family since
$$((\top_{X,X},1_X)\lda(f,f'))\rda(\top_{X,X},1_X)=(f',f)\rda(\top_{X,X},1_X)=(f,f').$$

Therefore, the assignment
\begin{equation} \label{Q_QG_embedding}
(f:X\to Y)\mapsto((f,\top_{Y,X}):X\to Y)
\end{equation}
defines a faithful quantaloid homomorphism that embeds $\CQ$ in the Girard quantaloid $\CQ_G$.
\end{proof}

In the case that $\CQ$ is a unital quantale, Proposition \ref{Q_QG} reduces to \cite[Theorem 6.1.3]{Rosenthal1990}, and the construction of $\CQ_G$ becomes the standard Chu construction over $\CQ$ (as a monoidal closed category).

\section{Chu connections and back diagonals between $\CQ$-distributors} \label{ChuConQDist}

In this section we are concerned with Chu connections and back diagonals in a special closed bicategory, i.e., the bicategory of distributors between categories enriched in a small quantaloid.

\subsection{Quantaloid-enriched categories} \label{Q_Cat}

In order to avoid size issues, from now on a \emph{small} quantaloid $\CQ$ is fixed as a base category for enrichment, and we shall use the notations of (small) $\CQ$-categories, $\CQ$-distributors and $\CQ$-functors mostly as in \cite{Shen2014,Shen2013a}. For the convenience of the readers, we take a quick tour of the preliminaries in this subsection.

A (small) \emph{$\CQ$-category} $\bbA$ consists of a set $\bbA_0$ as objects, a \emph{type} map $t:\bbA_0\to\ob\CQ$, and hom-arrows $\bbA(x,y)\in\CQ(tx,ty)$ such that $1_{tx}\leq\bbA(x,x)$ and $\bbA(y,z)\circ\bbA(x,y)\leq\bbA(x,z)$ for all $x,y,z\in\bbA_0$. $\bbB$ is a (full) \emph{$\CQ$-subcategory} of $\bbA$ if $\bbB_0\subseteq\bbA_0$ and $\bbB(x,y)=\bbA(x,y)$ whenever $x,y\in\bbB_0$.

For a $\CQ$-category $\bbA$, the underlying (pre)order on $\bbA_0$ is given by
$$x\leq y\iff tx=ty=X\ \text{and}\ 1_X\leq\bbA(x,y);$$
$\bbA$ is \emph{skeletal} if the underlying order on $\bbA_0$ is a partial order. $\bbA$ is called \emph{order-complete} if each $\bbA_X$, the $\CQ$-subcategory of $\bbA$ with all the objects of type $X\in\ob\CQ$, admits all joins in the underlying order.

A \emph{$\CQ$-distributor} $\phi:\bbA\oto\bbB$ between $\CQ$-categories is a map that assigns to each pair $(x,y)\in\bbA_0\times\bbB_0$ a $\CQ$-arrow $\phi(x,y)\in\CQ(tx,ty)$, such that $\bbB(y,y')\circ\phi(x,y)\circ\bbA(x',x)\leq\phi(x',y')$ for all $x,x'\in\bbA_0$, $y,y'\in\bbB_0$. With the pointwise order inherited from $\CQ$, the locally ordered 2-category $\QDist$ of $\CQ$-categories and $\CQ$-distributors is in fact a (large) quantaloid in which
\begin{align*}
&\psi\circ\phi:\bbA\oto\bbC,\quad(\psi\circ\phi)(x,z)=\bv_{y\in\bbB_0}\psi(y,z)\circ\phi(x,y),\\
&\xi\lda\phi:\bbB\oto\bbC,\quad(\xi\lda\phi)(y,z)=\bw_{x\in\bbA_0}\xi(x,z)\lda\phi(x,y),\\
&\psi\rda\xi:\bbA\oto\bbB,\quad (\psi\rda\xi)(x,y)=\bw_{z\in\bbC_0}\psi(y,z)\rda\xi(x,z)
\end{align*}
for $\CQ$-distributors $\phi:\bbA\oto\bbB$, $\psi:\bbB\oto\bbC$, $\xi:\bbA\oto\bbC$; the identity $\CQ$-distributor on $\bbA$ is given by hom-arrows $\bbA:\bbA\oto\bbA$. Adjoint $\CQ$-distributors are exactly adjoint arrows in the quantaloid $\QDist$.

A \emph{$\CQ$-functor} (resp. \emph{fully faithful} $\CQ$-functor) $F:\bbA\to\bbB$ between $\CQ$-categories is a map $F:\bbA_0\to\bbB_0$ such that $tx=t(Fx)$ and $\bbA(x,y)\leq\bbB(Fx,Fy)$ (resp. $\bbA(x,y)=\bbB(Fx,Fy)$) for all $x,y\in\bbA_0$. With the pointwise order of $\CQ$-functors
$$F\leq G:\bbA\to\bbB\iff\forall x\in\bbA_0:\ Fx\leq Gx\ \text{in}\ \bbB_0,$$
$\CQ$-categories and $\CQ$-functors constitute a locally ordered 2-category $\QCat$. Adjoint $\CQ$-functors are exactly adjoint 1-cells in $\QCat$, while fully faithful and bijective $\CQ$-functors are isomorphisms in $\QCat$.

Each $\CQ$-functor $F:A\to B$ induces an adjunction $F_{\nat}\dv F^{\nat}$ in $\QDist$ given by
\begin{align*}
&F_{\nat}:\bbA\oto\bbB,\quad F_{\nat}(x,y)=\bbB(Fx,y),\\
&F^{\nat}:\bbB\oto\bbA,\quad F^{\nat}(y,x)=\bbB(y,Fx),
\end{align*}
which are both 2-functorial as
$$(-)_{\nat}:(\QCat)^{\co}\to\QDist,\quad(-)^{\nat}:(\QCat)^{\op}\to\QDist,$$
where ``$\co$'' refers to the dualization of 2-cells.

\begin{prop} {\rm\cite{Shen2013a}} \label{fully_faithful_graph}
A $\CQ$-functor $F:\bbA\to\bbB$ is fully faithful if and only if $F^{\nat}\circ F_{\nat}=\bbA$.
\end{prop}

\begin{prop} {\rm\cite{Stubbe2005}} \label{adjoint_graph}
Let $F:\bbA\to\bbB$, $G:\bbB\to\bbA$ be a pair of $\CQ$-functors. Then
$$F\dv G\ \text{in}\ \QCat\iff F_{\nat}=G^{\nat}\iff G_{\nat}\dv F_{\nat}\ \text{in}\ \QDist\iff G^{\nat}\dv F^{\nat}\ \text{in}\ \QDist.$$
\end{prop}

\begin{rem}  \label{Qcat_dual}
The dual of a $\CQ$-category $\bbA$ is a $\CQ^{\op}$-category, given by $\bbA^{\op}_0=\bbA_0$ and $\bbA^{\op}(x,y)=\bbA(y,x)$ for all $x,y\in\bbA_0$. Each $\CQ$-functor $F:\bbA\to\bbB$ becomes a $\CQ^{\op}$-functor $F^{\op}:\bbA^{\op}\to\bbB^{\op}$ with the same mapping on objects but $(F')^{\op}\leq F^{\op}$ whenever $F\leq F':\bbA\to\bbB$. Each $\CQ$-distributor $\phi:\bbA\oto\bbB$ corresponds bijectively to a $\CQ^{\op}$-distributor $\phi^{\op}:\bbB^{\op}\oto\bbA^{\op}$ with $\phi^{\op}(y,x)=\phi(x,y)$ for all $x\in\bbA_0$, $y\in\bbB_0$. Therefore, as already noted in \cite{Stubbe2005}, one has a 2-isomorphism
$$(-)^{\op}:\QCat\cong(\CQ^{\op}\text{-}\Cat)^{\co}$$
and an isomorphism of quantaloids
$$(-)^{\op}:\QDist\cong(\CQ^{\op}\text{-}\Dist)^{\op}.$$
We note in passing that our terminologies of quantaloid-enriched categories are not exactly the same as  in our  main references,  \cite{Stubbe2005,Stubbe2006}, on the subject: 
the $\CQ$-categories and $\CQ$-distributors here are exactly $\CQ^{\op}$-categories and $\CQ^{\op}$-distributors there.
\end{rem}

A \emph{presheaf} with type $X$ on a $\CQ$-category $\bbA$ is a $\CQ$-distributor $\mu:\bbA\oto *_X$, where $*_X$ is the $\CQ$-category with only one object of type $X$. Presheaves on $\bbA$ constitute a $\CQ$-category $\PA$ with $\PA(\mu,\mu')=\mu'\lda\mu$ for all $\mu,\mu'\in\PA$. Dually, the $\CQ$-category $\PdA$ of \emph{copresheaves} on $\bbA$ consists of $\CQ$-distributors $\lam:*_X\oto\bbA$ as objects with type $X$ and $\PdA(\lam,\lam')=\lam'\rda\lam$ for all $\lam,\lam'\in\PdA$. It is easy to see that $\PdA\cong(\PA^{\op})^{\op}$.

A $\CQ$-category $\bbA$ is \emph{complete} if the \emph{Yoneda embedding} $\sY:\bbA\to\PA,\ x\mapsto\bbA(-,x)$ has a left adjoint $\sup:\PA\to\bbA$ in $\QCat$; that is,
$$\bbA(\sup\mu,-)=\PA(\mu,\sY_{\bbA}-)=\bbA\lda\mu$$
for all $\mu\in\PA$. It is well known that $\bbA$ is a complete $\CQ$-category if and only if $\bbA^{\op}$ is a complete $\CQ^{\op}$-category \cite{Stubbe2005},
where the completeness of $\bbA^{\op}$ may be translated as the \emph{co-Yoneda embedding} $\sYd:\bbA\oto\PdA,\ x\mapsto\bbA(x,-)$ admitting a right adjoint $\inf:\PdA\to\bbA$ in $\QCat$.

\begin{lem}[Yoneda] {\rm\cite{Stubbe2005}} \label{Yoneda_lemma}
Let $\bbA$ be a $\CQ$-category and $\mu\in\PA$, $\lam\in\PdA$. Then
$$\mu=\PA(\sY-,\mu)=\sY_{\nat}(-,\mu),\quad\lam=\PdA(\lam,\sYd-)=(\sYd)^{\nat}(\lam,-).$$
In particular, both $\sY$ and $\sYd$ are fully faithful $\CQ$-functors.
\end{lem}

In a $\CQ$-category $\bbA$, the \emph{tensor} of $f\in\CP(tx)$ and $x\in\bbA_0$ (here $\CP(tx)$ stands for $\CP *_{tx}$, and $f\in\CP(tx)$ is essentially a $\CQ$-arrow with domain $tx$) , denoted by $f\otimes x$, is an object in $\bbA_0$ of type $t(f\otimes x)=tf$ such that $\bbA(f\otimes x,-)=\bbA(x,-)\lda f$. $\bbA$ is \emph{tensored} if $f\otimes x$ exists for all choices of $f$ and $x$; $\bbA$ is \emph{cotensored} if $\bbA^{\op}$ is tensored.

\begin{thm} {\rm\cite{Stubbe2006}} \label{complete_tensor}
A $\CQ$-category $\bbA$ is complete if, and only if, $\bbA$ is tensored, cotensored and order-complete.
\end{thm}

\begin{exmp} \label{PA_tensor_complete} {\rm\cite{Stubbe2006}}
For each $\CQ$-category $\bbA$, $\PA$ and $\PdA$ are both skeletal, tensored, cotensored and complete $\CQ$-categories. In particular, tensors in $\PA$ are given by $f\otimes\mu=f\circ\mu$ for all $\mu\in\PA$ and $f\in\CP(t\mu)$.
\end{exmp}

\begin{prop} {\rm\cite{Stubbe2005,Stubbe2006}} \label{la_condition}
Let $F:\bbA\to\bbB$ be a $\CQ$-functor, with $\bbA$ complete. The following statements are equivalent:
\begin{itemize}
\item[\rm (i)] $F$ is a left adjoint in $\QCat$.
\item[\rm (ii)] $F$ is \emph{sup-preserving} in the sense that $F({\sup}_{\bbA}\mu)={\sup}_{\bbB}(\mu\circ F^{\nat})$ for all $\mu\in\PA$.
\item[\rm (iii)] $F$ is a left adjoint between the underlying ordered sets of $\bbA$, $\bbB$, and preserves tensors in the sense that $F(f\otimes_{\bbA}x)= f\otimes_{\bbB}Fx$ for all $x\in\bbA_0$, $f\in\CP(tx)$.
\end{itemize}
\end{prop}

Skeletal complete $\CQ$-categories and left adjoint $\CQ$-functors (or equivalently, sup-preserving $\CQ$-functors) constitute a 2-subcategory of $\QCat$ and we denote it by $\QCCat$. Indeed, it is easy to check that $\QCCat$ is a (large) quantaloid, in which the join of $\{F_i\}_{i\in I}\subseteq\QCCat(\bbA,\bbB)$ is the same as in $\QCat(\bbA,\bbB)$.

From the 2-isomorphism $\QCat\cong(\CQ^{\op}\text{-}\Cat)^{\co}$ in Remark \ref{Qcat_dual} it is easy to see that a left adjoint $F:\bbA\to\bbB$ in $\QCat$ corresponds bijectively to a right adjoint $F^{\op}:\bbB^{\op}\to\bbA^{\op}$ in $\CQ^{\op}\text{-}\Cat$, thus one soon obtains the following isomorphism of quantaloids:

\begin{prop} \label{QCCat_QopCCatop_iso}
Mapping a left adjoint $\CQ$-functor $F:\bbA\to\bbB$ between skeletal complete $\CQ$-categories to $G^{\op}:\bbB^{\op}\to\bbA^{\op}$ with $G$ the right adjoint of $F$ in $\QCat$ induces an isomorphism of quantaloids
$$\dv{}^{\op}:\QCCat\to(\CQ^{\op}\text{-}\CCat)^{\op}.$$
\end{prop}

Each $\CQ$-distributor $\phi:\bbA\oto\bbB$ induces an \emph{Isbell adjunction} \cite{Shen2013a} $\uphi\dv\dphi$ in $\QCat$ given by
\begin{align*}
&\uphi:\PA\to\PdB,\quad\mu\mapsto\phi\lda\mu,\\
&\dphi:\PdB\to\PA,\quad\lam\mapsto\lam\rda\phi;
\end{align*}
and a \emph{Kan adjunction} $\phi^*\dv\phi_*$ defined as
\begin{align*}
&\phi^*:\PB\to\PA,\quad \lam\mapsto\lam\circ\phi,\\
&\phi_*:\PA\to\PB,\quad \mu\mapsto\mu\lda\phi.
\end{align*}
We also write down the dual Kan adjunction $\phi_{\dag}\dv\phi^{\dag}$:
\begin{align*}
&\phi_{\dag}:=((\phi^{\op})_*)^{\op}:\PdB\to\PdA,\quad\lam\mapsto\phi\rda\lam,\\
&\phi^{\dag}:=((\phi^{\op})^*)^{\op}:\PdA\to\PdB,\quad\mu\mapsto\phi\circ\mu,
\end{align*}
which corresponds to the Kan adjunction $(\phi^{\op})^*\dv(\phi^{\op})_*:\PB^{\op}\to\PA^{\op}$ in $\CQ^{\op}\text{-}\Cat$ under the isomorphism in Proposition \ref{QCCat_QopCCatop_iso}.

\begin{exmp}
The contravariant Galois connection $\uR\dv\dR$ in Example \ref{ChuConRel} is exactly an Isbell adjunction induced by $R\subseteq A\times B$ (considered as a ${\bf 2}$-distributor between discrete ${\bf 2}$-categories, i.e., sets). The induced Kan adjunction $R^*\dv R_*$ is obviously the Galois connection with the same symbol in Example \ref{SqRel}.
\end{exmp}

\begin{prop} {\rm\cite{Heymans2010}} \label{star_graph_adjoint}
$(-)^*:\QDist\to(\QCat)^{\op}$ and $(-)^{\dag}:\QDist\to(\QCat)^{\co}$ are both 2-functorial, and one has two pairs of adjoint 2-functors
$$(-)^*\dv(-)^{\nat}:(\QCat)^{\op}\to\QDist\quad\text{and}\quad(-)_{\nat}\dv(-)^{\dag}:\QDist\to(\QCat)^{\co}.$$
\end{prop}

The adjunctions $(-)^*\dv(-)^{\nat}$ and $(-)_{\nat}\dv(-)^{\dag}$ give rise to isomorphisms
$$(\QCat)^{\co}(\bbA,\PdB)\cong\QDist(\bbA,\bbB)\cong\QCat(\bbB,\PA)$$
for all $\CQ$-categories $\bbA$, $\bbB$. We denote by
\begin{align}
&\tphi:\bbB\to\PA,\quad\tphi y=\phi(-,y),\label{tphi_def}\\
&\hphi:\bbA\to\PdB,\quad\hphi x=\phi(x,-) \label{hphi_def}
\end{align}
for the transposes of each $\CQ$-distributor $\phi:\bbA\oto\bbB$.

\begin{prop} \label{tphi_hphi_Yoneda} {\rm\cite{Shen2014,Shen2013a}}
Let $\phi:\bbA\oto\bbB$ be a $\CQ$-distributor. Then for all $x\in\bbA_0$, $y\in\bbB_0$,
$$\hphi x=\phi(x,-)=\uphi\sY_{\bbA}x=\phi^{\dag}\sYd_{\bbA}x\quad\text{and}\quad\tphi y=\phi(-,y)=\dphi\sYd_{\bbB}y=\phi^*\sY_{\bbB}y.$$
\end{prop}

Since the ``$\CQ$-natural transformation'' between $\CQ$-functors is simply given by the local order in $\QCat$, a \emph{$\CQ$-monad} on a $\CQ$-category $\bbA$ is exactly a $\CQ$-functor $F:\bbA\to\bbA$ with $1_{\bbA}\leq F$ and $F^2\cong F$. A \emph{$\CQ$-comonad} on $\bbA$ may be defined through a $\CQ^{\op}$-monad on $\bbA^{\op}$.\footnote{$\CQ$-monads (resp. $\CQ$-comonads) are referred to as $\CQ$-closure operators (resp. $\CQ$-interior operators) in \cite{Shen2014,Shen2013a}, due to their similarities to corresponding terminologies in topology.}

Each adjunction $F\dv G:\bbB\to\bbA$ in $\QCat$ gives rise to a $\CQ$-monad $GF$ on $\bbA$ and a $\CQ$-comonad $FG$ on $\bbB$. In particular, for each $\CQ$-distributor $\phi:\bbA\oto\bbB$, $\dphi\uphi$ (resp. $\uphi\dphi$) is an idempotent $\CQ$-monad (resp. $\CQ$-comonad) on $\PA$ (resp. $\PdA$) since $\PA$ (resp. $\PdA$) is a skeletal $\CQ$-category.


\begin{prop} \label{monad_reflective} {\rm\cite{Shen2013a}}
Suppose  $F:\bbA\to\bbA$ is a $\CQ$-monad (resp. $\CQ$-comonad) on a skeletal $\CQ$-category $\bbA$. Let
$$\Fix(F):=\{x\in\bbA_0\mid Fx=x\}=\{Fx\mid x\in\bbA_0\}$$
be the $\CQ$-subcategory of $\bbA$ consisting of the fixed points of $F$. Then
\begin{itemize}
\item[\rm (1)] the inclusion $\CQ$-functor $\Fix(F)\ \to/^(->/\bbA$ is right (resp. left) adjoint to the codomain restriction $F:\bbA\to\Fix(F)$;
\item[\rm (2)] $\Fix(F)$ is a complete $\CQ$-category provided so is $\bbA$.
\end{itemize}
\end{prop}

If $\bbB=\Fix(F)$ for a $\CQ$-monad $F:\bbA\to\bbA$, suprema in $\bbB$ are given by ${\sup}_{\bbB}\mu=F{\sup}_{\bbA}(\mu\circ I^{\nat})$ for all $\mu\in\PB$. In particular, for all $x,x_i\in\bbB_0$ $(i\in I)$, $f\in\CP(tx)$,
\begin{equation} \label{tensor_closure_system}
f\otimes_{\bbB}x=F(f\otimes_{\bbA}x),\quad\bigsqcup_{i\in I}x_i=F\Big(\bv_{i\in I}x_i\Big),
\end{equation}
where $\bigsqcup$ and $\bv$ respectively denote the underlying joins in $\bbB$ and $\bbA$.


\subsection{Chu connections in the quantaloid $\QDist$}

Given $\CQ$-distributors $\phi:\bbA\oto\bbB$ and $\psi:\bbA'\oto\bbB'$, a \emph{Chu connection} from $\phi$ to $\psi$ is, by definition, a pair $(\zeta:\bbA\oto\bbA',\ \eta:\bbB\oto\bbB')$ of $\CQ$-distributors with
$$\phi\lda\zeta=\eta\rda\psi.
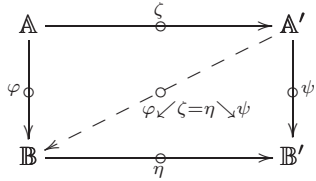
$$
$$\bfig
\square<1000,500>[\bbA`\bbA'`\bbB`\bbB';\zeta`\phi`\psi`\eta]
\morphism(1000,500)|r|/-->/<-1000,-500>[\bbA'`\bbB;\ \phi\lda\zeta=\eta\rda\psi]
\place(0,250)[\circ] \place(500,250)[\circ] \place(1000,250)[\circ] \place(500,0)[\circ] \place(500,500)[\circ]
\efig$$

Chu connections between $\CQ$-distributors are natural extensions of \emph{Chu transforms}. A Chu transform (called infomorphism in \cite{Shen2014,Shen2013a}) $$(F,G):(\phi:\bbA\oto\bbB)\to(\psi:\bbA'\oto\bbB')$$
is a pair of $\CQ$-functors $F:\bbA\to\bbA'$ and $G:\bbB'\to\bbB$ such that $\psi(F-,-)=\phi(-,G-)$.

\begin{prop} \label{QCat_Chu_transform}
Given a pair of $\CQ$-functors $F:\bbA\to\bbA'$ and $G:\bbB'\to\bbB$, the following statements are equivalent:
\begin{itemize}
\item[\rm (i)] $(F,G):\phi\to\psi$ is a Chu transform.
\item[\rm (ii)] $\psi\circ F_{\nat}=G^{\nat}\circ\phi$; that is, $(F_{\nat},G^{\nat}):\phi\to\psi$ is an arrow in the quantaloid $\Arr(\QDist)$.
\item[\rm (iii)] $\psi\lda F^{\nat}=G_{\nat}\rda\phi$; that is, $(F^{\nat},G_{\nat}):\psi\to\phi$ is a Chu connection.
\item[\rm (iv)] $\upsi(F^{\nat})^*=(G^{\nat})^{\dag}\uphi$.
\item[\rm (v)] $(F_{\nat})^*\psi^*=\phi^*(G^{\nat})^*$.
$$\bfig
\square(-300,0)<1000,500>[\bbA`\bbA'`\bbB`\bbB';F_{\nat}`\phi`\psi`G^{\nat}]
\morphism(-300,500)|b|/-->/<1000,-500>[\bbA`\bbB';\psi\circ F_{\nat}=G^{\nat}\circ\phi]
\place(200,0)[\circ] \place(200,500)[\circ] \place(-300,250)[\circ] \place(200,250)[\circ] \place(700,250)[\circ]
\square(1500,0)<1000,500>[\bbA'`\bbA`\bbB'`\bbB;F^{\nat}`\psi`\phi`G_{\nat}]
\morphism(2500,500)|b|/-->/<-1000,-500>[\bbA`\bbB';\psi\lda F^{\nat}=G_{\nat}\rda\phi]
\place(2000,0)[\circ] \place(2000,500)[\circ] \place(1500,250)[\circ]
\place(2500,250)[\circ] \place(2000,250)[\circ]
\square(0,-800)<700,500>[\PA`\PA'`\PdB`\PdB';(F^{\nat})^*`\uphi`\upsi`(G^{\nat})^{\dag}]
\square(1500,-800)<700,500>[\PB'`\PB`\PA'`\PA;(G^{\nat})^*`\psi^*`\phi^*`(F_{\nat})^*]
\efig$$
\end{itemize}
\end{prop}

\begin{proof}
(i)$\iff$(ii): Straightforward calculation by the definition of composite $\CQ$-distributors.

(ii)$\iff$(iii): Follows immediately from Proposition \ref{adjoint_arrow_calculation}(1).

(ii)${}\Lra{}$(iv): For each $\mu\in\PA$,
\begin{align*}
\upsi(F^{\nat})^*\mu&=\psi\lda(\mu\circ F^{\nat})\\
&=(\psi\circ F_{\nat})\lda\mu&(\text{Proposition \ref{adjoint_arrow_calculation}(2)})\\
&=(G^{\nat}\circ\phi)\lda\mu\\
&=G^{\nat}\circ(\phi\lda\mu)&(\text{Proposition \ref{adjoint_arrow_calculation}(3)})\\
&=(G^{\nat})^{\dag}\uphi\mu.
\end{align*}

(iv)${}\Lra{}$(iii): For each $x\in\bbA_0$,
\begin{align*}
(\psi\lda F^{\nat})(x,-)&=\psi\lda F^{\nat}(-,x)\\
&=\upsi(F^{\nat})^*\sY_{\bbA}x&\text{(Proposition \ref{tphi_hphi_Yoneda})}\\
&=(G^{\nat})^{\dag}\uphi\sY_{\bbA}x\\
&=G^{\nat}\circ\phi(x,-)&\text{(Proposition \ref{tphi_hphi_Yoneda})}\\
&=(G_{\nat}\rda\phi)(x,-).&\text{(Proposition \ref{adjoint_arrow_calculation}(1))}
\end{align*}

(ii)${}\Lra{}$(v): Follows immediately from the functoriality of $(-)^*:\QDist\to(\QCat)^{\op}$.

(v)${}\Lra{}$(ii): For each $y'\in\bbB'_0$,
$$(\psi\circ F_{\nat})(-,y')=(F_{\nat})^*\psi^*\sY_{\bbB'}y'=\phi^*(G^{\nat})^*\sY_{\bbB'}y'=(G^{\nat}\circ\phi)(-,y'),$$
where the first and the last equalities follow from Proposition \ref{tphi_hphi_Yoneda}.
\end{proof}

Chu transforms between $\CQ$-distributors may be ordered as
$$(F,G)\leq(F',G'):\phi\to\psi\iff F\leq F'\ \text{and}\ G\leq G'.$$
In this way $\QChu$ becomes a locally ordered 2-category (denoted by $\QInfo$ in \cite{Shen2014,Shen2013a}).

Proposition \ref{QCat_Chu_transform} gives rise to functors\footnote{Here ``$\ChuCon(\QDist)^{\op}$'' should read ``the opposite of the quantaloid $\ChuCon(\QDist)$''.}
\begin{align*}
&(\Box_{\nat},\Box^{\nat}):\QChu\to\Arr(\QDist),\quad(F,G)\mapsto(F_{\nat},G^{\nat}),\\
&(\Box^{\nat},\Box_{\nat}):\QChu\to\ChuCon(\QDist)^{\op},\quad(F,G)\mapsto(F^{\nat},G_{\nat}),
\end{align*}
which are both identities on objects, but neither of them is full, faithful or 2-functorial.

It is interesting that Chu connections can be characterized as Chu transforms:

\begin{prop} \label{Chu_con_condition}
Let $\phi:\bbA\oto\bbB$, $\psi:\bbA'\oto\bbB'$, $\zeta:\bbA\oto\bbA'$ and $\eta:\bbB\oto\bbB'$ be $\CQ$-distributors. The following statements are equivalent:
\begin{itemize}
\item[\rm (i)] $(\zeta, \eta):\phi\to\psi$ is a Chu connection.
\item[\rm (ii)] $\uphi\zeta^*=\eta_{\dag}\upsi$.
\item[\rm (iii)] $(\zeta^*, \eta^\dag):(\upsi)_{\nat}\to(\uphi)_{\nat}$ is a Chu transform.
\end{itemize}
$$\bfig
\square<700,500>[\PA'`\PA`\PdB'`\PdB;\zeta^*`\upsi`\uphi`\eta_{\dag}]
\square(1500,0)/->`->`->`<-/<700,500>[\PA'`\PA`\PdB'`\PdB;\zeta^*`(\upsi)_{\nat}`(\uphi)_{\nat}`\eta^\dag]
\place(1500,250)[\circ] \place(2200,250)[\circ]
\efig$$
Therefore, there is a functor
$$(\Box^*,\Box^{\dag}):\ChuCon(\QDist)^{\op}\to\QChu$$
that sends a Chu connection $(\zeta,\eta):\phi\to\psi$ to the Chu transform $(\zeta^*, \eta^\dag):(\upsi)_{\nat}\to(\uphi)_{\nat}$.
\end{prop}

\begin{proof}
(i)${}\Lra{}$(ii): For all $\mu'\in\PA'$,
\begin{align*}
\uphi\zeta^*\mu'&=\phi\lda(\mu'\circ\zeta)\\
&=(\phi\lda\zeta)\lda\mu'\\
&=(\eta\rda\psi)\lda\mu'&((\zeta, \eta)\ \text{is a Chu connection})\\
&=\eta\rda(\psi\lda\mu')\\
&=\eta_{\dag}\upsi\mu'.
\end{align*}

(ii)${}\Lra{}$(iii): Since $\uphi\zeta^*=\eta_{\dag}\upsi$, the functoriality of $(-)_{\nat}:(\QCat)^{\co}\to\QDist$ leads to
$$(\uphi)_{\nat}\circ(\zeta^*)_{\nat}=(\uphi\zeta^*)_{\nat}=(\eta_{\dag}\upsi)_{\nat}=(\eta_{\dag})_{\nat}\circ(\upsi)_{\nat}=(\eta^{\dag})^{\nat}\circ(\upsi)_{\nat}.$$
Here the last equality follows from Proposition \ref{adjoint_graph} and the dual Kan adjunction $\eta_{\dag}\dv\eta^{\dag}$.

(iii)${}\Lra{}$(i): $(\zeta^*, \eta^\dag):(\upsi)_{\nat}\to(\uphi)_{\nat}$ being a Chu transform implies
\begin{equation} \label{zeta_star_eta_dag}
\PdB(\uphi\zeta^*-,-)=(\uphi)_{\nat}(\zeta^*-,-)=(\upsi)_{\nat}(-,\eta^\dag-)=\PdB'(\upsi-,\eta^\dag-).
\end{equation}
Thus for all $x'\in\bbA'_0$ and $y\in\bbB_0$,
\begin{align*}
(\phi\lda\zeta)(x',y)&=\phi(-,y)\lda\zeta(-,x')\\
&=(\dphi\sYd_{\bbB}y)\lda(\zeta^*\sY_{\bbA'}x')&(\text{Proposition \ref{tphi_hphi_Yoneda}})\\
&=\PA(\zeta^*\sY_{\bbA'}x',\dphi\sYd_{\bbB}y)\\
&=\PdB(\uphi\zeta^*\sY_{\bbA'}x',\sYd_{\bbB}y)&(\uphi\dv\dphi)\\
&=\PdB'(\upsi\sY_{\bbA'}x',\eta^\dag\sYd_{\bbB}y)&\text{(Equation (\ref{zeta_star_eta_dag}))}\\
&=(\eta^\dag\sYd_{\bbB}y)\rda(\upsi\sY_{\bbA'}x')\\
&=\eta(y,-)\rda\psi(x',-)&(\text{Proposition \ref{tphi_hphi_Yoneda}})\\
&=(\eta\rda\psi)(x',y).
\end{align*}
\end{proof}


\subsection{The quantaloids $\ChuCon(\QDist)$, $\BB(\QDist)$ and $\QCCat$}

Each $\CQ$-distributor $\phi:\bbA\oto\bbB$ induces an Isbell adjunction $\uphi\dv\dphi$. It follows from Proposition \ref{monad_reflective} that
$$\CM\phi:=\Fix(\dphi\uphi),$$
the $\CQ$-subcategory of fixed points of the $\CQ$-monad $\dphi\uphi:\PA\to\PA$, is a complete $\CQ$-category. It is known \cite{Shen2013a} that the assignment $\phi\mapsto\CM\phi$ is an extension of the Dedekind-MacNeille completion of partially ordered sets and it is functorial from the category $\QChu$ to $\QCCat$, sending a Chu transform
$$(F,G):(\phi:\bbA\oto\bbB)\to(\psi:\bbA'\oto\bbB')$$
to the left adjoint $\CQ$-functor given by the composite
$$\CM(F,G)=(\CM\phi\ \to/^(->/\PA\to^{(F^{\nat})^*}\PA'\to^{\dpsi\upsi}\CM\psi).$$

The following proposition shows that the assignment $\phi\mapsto\CM\phi$ generates a contravariant functor
$$\CM:\ChuCon(\QDist)^{\op}\to\QCCat$$
that maps a Chu connection
$$(\zeta,\eta):(\phi:\bbA\oto\bbB)\to (\psi:\bbA'\oto\bbB')$$
to the left adjoint $\CQ$-functor
$$\CM(\zeta,\eta):=(\CM\psi\ \to/^(->/\PA'\to^{\zeta^*}\PA\to^{\dphi\uphi}\CM\phi).$$

\begin{prop} \label{M_functor}
$\CM:\ChuCon(\QDist)^{\op}\to\QCCat$ is a full functor. Moreover, $\CM$ is a quantaloid homomorphism.
\end{prop}

Before proving this proposition, we would like to point out that the composite of $\CM$ with the functor $(\Box^{\nat},\Box_{\nat}): \QChu\to\ChuCon(\QDist)^{\op}$ is exactly the functor $\CM:\QChu\to\QCCat$ obtained in \cite{Shen2013a}. So, the functor $\CM$ in Proposition \ref{M_functor} is an extension of the functor $\CM:\QChu\to\QCCat$ in \cite{Shen2013a}. 

We first prove the following lemma as a preparation:

\begin{lem} \label{zeta_continuous}
If $(\zeta,\eta):(\phi:\bbA\oto\bbB)\to(\psi:\bbA'\oto\bbB')$ is a Chu connection, then
$$\zeta^*\dpsi\upsi\leq\dphi\uphi\zeta^*:\PA'\to\PA.$$
\end{lem}

\begin{proof}
Consider the following diagram:
$$\bfig
\square|alrb|<700,500>[\PA'`\PdB'`\PA`\PdB;\upsi`\zeta^*`\eta_{\dag}`\uphi]
\square(700,0)/>``>`>/<700,500>[\PdB'`\PA'`\PdB`\PA;\dpsi``\zeta^*`\dphi]
\place(1050,250)[\twoar(-1,-1)]
\efig$$
Note that the commutativity of the left square follows from Proposition \ref{Chu_con_condition}(ii), and it suffices to prove $\zeta^*\dpsi\leq\dphi\eta_{\dag}$. Indeed, for all $\lam'\in\PdB'$,
\begin{align*}
\zeta^*\dpsi\lam'&=(\lam'\rda\psi)\circ\zeta\\
&\leq(\eta\rda\lam')\rda((\eta\rda\psi)\circ\zeta)\\
&=(\eta\rda\lam')\rda((\phi\lda\zeta)\circ\zeta)&((\zeta,\eta)\ \text{is a Chu connection})\\
&\leq(\eta\rda\lam')\rda\phi\\
&=\dphi\eta_{\dag}\lam'.
\end{align*}
\end{proof}

\begin{proof}[Proof of Proposition \ref{M_functor}]
{\bf Step 1.} $\CM(\zeta,\eta):\CM\psi\to\CM\phi$ is a left adjoint $\CQ$-functor.

First, we claim that $\zeta_*:\PA\to\PA'$ can be restricted as a $\CQ$-functor $\zeta_*:\CM\phi\to\CM\psi$. Indeed, from $\zeta^*\dv\zeta_*:\PA\to\PA'$ and Lemma \ref{zeta_continuous} one has $\dpsi\upsi\leq\zeta_*\dphi\uphi\zeta^*$, and consequently for all $\mu\in\CM\phi$,
$$\dpsi\upsi\zeta_*\mu\leq \zeta_*\dphi\uphi\zeta^*\zeta_*\mu\leq\zeta_*\dphi\uphi\mu =\zeta_*\mu.$$
This shows that $\zeta_*\mu\in\CM\psi$ for all $\mu\in\CM\phi$. So, it remains to prove $\CM(\zeta,\eta)\dv\zeta_*:\CM\phi\to\CM\psi$. Since $\dphi\uphi$ is a $\CQ$-monad on $\PA$,  then for all $\mu'\in\CM\psi$, $\mu\in\CM\phi$,
\begin{align*}
\PA(\zeta^*\mu',\mu)&\leq\PA(\dphi\uphi\zeta^*\mu',\dphi\uphi\mu)\\
&=\PA(\dphi\uphi\zeta^*\mu',\mu)\\
&=\mu\lda(\dphi\uphi\zeta^*\mu')\\
&\leq\mu\lda\zeta^*\mu'\\
&=\PA(\zeta^*\mu',\mu).
\end{align*}
Hence
\begin{align*}
\CM\phi(\CM(\zeta,\eta)\mu',\mu)&=\PA(\dphi\uphi\zeta^*\mu',\mu)\\
&=\PA(\zeta^*\mu',\mu)\\
&=\PA'(\mu',\zeta_*\mu)\\
&=\CM\psi(\mu',\zeta_*\mu),
\end{align*}
as desired.

{\bf Step 2.} $\CM:\ChuCon(\QDist)^{\op}\to\QCCat$ is a functor. For this one must check that
$$\CM(\zeta,\eta)\CM(\zeta',\eta')= \CM(\zeta'\circ\zeta,\eta'\circ\eta):\CM\xi\to\CM\phi$$
for any Chu connections $(\zeta,\eta):\phi\to \psi$ and $(\zeta',\eta'):\psi\to \xi$. It suffices to show that
$$\dphi\uphi\zeta^*\dpsi\upsi\zeta'^* = \dphi\uphi\zeta^*\zeta'^*.$$  On one hand, by Lemma \ref{zeta_continuous} one immediately has
$$\dphi\uphi\zeta^*\dpsi\upsi\zeta'^*\leq\dphi\uphi\dphi\uphi\zeta^*\zeta'^* =\dphi\uphi\zeta^*\zeta'^*$$
since $\dphi\uphi$ is idempotent. On the other hand, $\dphi\uphi\zeta^*\zeta'^* \leq\dphi\uphi\zeta^*\dpsi\upsi\zeta'^*$ holds trivially since $1_{\PA'}\leq\dpsi\upsi$.

{\bf Step 3.} $\CM$ is full. For all $\CQ$-distributors $\phi:\bbA\oto\bbB$, $\psi:\bbA'\oto\bbB'$, one needs to show that
$$\CM:\ChuCon(\QDist)(\phi,\psi)\to\QCCat(\CM\psi,\CM\phi)$$
is surjective.

For a left adjoint $\CQ$-functor $F:\CM\psi\to\CM\phi$, let $G:\CM\phi\to\CM\psi$ be its right adjoint. Define $\CQ$-distributors $\zeta:\bbA\oto\bbA'$, $\eta:\bbB\oto\bbB'$ through their transposes (see Equations (\ref{tphi_def}), (\ref{hphi_def}) for the definition):
\begin{align}
\tzeta:=&(\bbA'\to^{\sY_{\bbA'}}\PA'\to^{\dpsi\upsi}\CM\psi\to^F\CM\phi\ \to/^(->/\PA),\label{tzeta_def}\\
\heta:=&(\bbB\to^{\sYd_{\bbB}}\PdB\to^{\dphi}\CM\phi\to^G\CM\psi\to^{\upsi}\Fix(\upsi\dpsi)\ \to/^(->/\PdB').\label{heta_def}
\end{align}
We claim that $(\zeta,\eta):\phi\to\psi$ is a Chu connection and $\CM(\zeta,\eta)=F$.

(1) $(\zeta,\eta):\phi\to\psi$ is a Chu connection. For all $x'\in\bbA'_0$ and $y\in\bbB_0$ it holds that
\begin{align*}
(\phi\lda\zeta)(x',y)&=\phi(-,y)\lda\tzeta x'&\text{(Equation (\ref{tphi_def}))}\\
&=(\dphi\sYd_{\bbB}y)\lda(F\dpsi\upsi\sY_{\bbA'}x')&\text{(Equation (\ref{tzeta_def}))}\\
&=\PA(F\dpsi\upsi\sY_{\bbA'}x',\ \dphi\sYd_{\bbB}y)\\
&=\PA'(\dpsi\upsi\sY_{\bbA'}x',\ G\dphi\sYd_{\bbB}y)&(F\dv G)\\
&=\PA'(\dpsi\upsi\sY_{\bbA'}x',\ \dpsi\upsi G\dphi\sYd_{\bbB}y)&(\text{the codomain of}\ G\ \text{is}\ \CM\psi)\\
&=\PdB'(\upsi\dpsi\upsi\sY_{\bbA'}x',\ \upsi G\dphi\sYd_{\bbB}y)&(\upsi\dv\dpsi)\\
&=\PdB'(\upsi\sY_{\bbA'}x',\ \upsi G\dphi\sYd_{\bbB}y)\\
&=(\upsi G\dphi\sYd_{\bbB}y)\rda(\upsi\sY_{\bbA'}x')\\
&=\heta y\rda\psi(x',-)&\text{(Equation (\ref{heta_def}))}\\
&=(\eta\rda\psi)(x',y),&\text{(Equation (\ref{hphi_def}))}
\end{align*} showing that $(\zeta,\eta):\phi\to\psi$ is a Chu connection.

(2) $\CM(\zeta,\eta)=F$. First of all, it follows from Example \ref{PA_tensor_complete} and Equation (\ref{tensor_closure_system}) that the tensor $f\otimes\mu$ in $\CM\phi$ is given by
\begin{equation} \label{tensor_DPB}
f\otimes\mu=\dphi\uphi(f\circ\mu)
\end{equation}
for all $\mu\in \CM\phi$, $f\in\CP(t\mu)$. Note that $\dphi\uphi:\PA\to\CM\phi$ and $\dpsi\upsi:\PA'\to \CM\psi$ are both left adjoint $\CQ$-functors by Proposition \ref{monad_reflective}(1), thus so is the composite
$$F\cdot\dpsi\upsi:\PA'\to \CM\psi\to \CM\phi.$$

For any $\mu'\in \CM\psi$, since the presheaf $\mu'\circ\zeta$ is the pointwise join of the $\CQ$-distributors $\mu'(x')\circ(\tzeta x')$ $(x'\in\bbA'_0)$, one has
\begin{align*}
\CM(\zeta,\eta)\mu'&=\dphi\uphi\zeta^*\mu'\\
&=\dphi\uphi(\mu'\circ\zeta)\\
&=\dphi\uphi\Big(\bv_{x'\in\bbA'_0}\mu'(x')\circ(\tzeta x')\Big)\\
&=\dphi\uphi\Big(\bv_{x'\in\bbA'_0}\mu'(x')\circ(F\dpsi\upsi\sY_{\bbA'}x')\Big)&\text{(Equation (\ref{tzeta_def}))}\\
&=\bigsqcup_{x'\in\bbA'_0}\dphi\uphi(\mu'(x')\circ(F\dpsi\upsi\sY_{\bbA'}x'))&\text{(Proposition \ref{la_condition}(iii))}\\
&=\bigsqcup_{x'\in\bbA'_0}\mu'(x')\otimes(F\dpsi\upsi\sY_{\bbA'}x')&\text{(Equation (\ref{tensor_DPB}))}\\
&=F\dpsi\upsi\Big(\bv_{x'\in\bbA'_0}\mu'(x')\circ(\sY_{\bbA'}x')\Big)&\text{(Proposition \ref{la_condition}(iii))}\\
&=F\dpsi\upsi(\mu'\circ\bbA')\\
&=F\dpsi\upsi\mu'\\
&=F\mu',&(\mu'\in\CM\psi)
\end{align*}
where $\bv$ and $\bigsqcup$ respectively denote the underlying joins in $\PA$ and $\CM\phi$. Therefore $\CM(\zeta,\eta)=F$, as desired.

{\bf Step 4.} $\CM:\ChuCon(\QDist)^{\op}\to\QCCat$ is a quantaloid homomorphism. To show that $\CM$ preserves joins of Chu connections, let $\{(\zeta_i,\eta_i)\}_{i\in I}$ be a family of Chu connections from $\phi:\bbA\oto\bbB$ to $\psi:\bbA'\oto\bbB'$, one must check that
$$\CM\Big(\bv_{i\in I}\zeta_i,\bv_{i\in I}\eta_i\Big)=\bigsqcup_{i\in I}\CM(\zeta_i,\eta_i):\CM\psi\to\CM\phi,$$
where $\bigsqcup$ denotes the pointwise join in $\QCCat(\CM\psi,\CM\phi)$ inherited from $\CM\phi$. Indeed, since $\dphi\uphi:\PA\to\CM\phi$ is a left adjoint $\CQ$-functor, one has
\begin{align*}
\CM\Big(\bv_{i\in I}\zeta_i,\bv_{i\in I}\eta_i\Big)\mu'&=\dphi\uphi\Big(\bv_{i\in I}\zeta_i\Big)^*\mu'\\
&=\dphi\uphi\Big(\mu'\circ\bv_{i\in I}\zeta_i\Big)\\
&=\dphi\uphi\Big(\bv_{i\in I}\mu'\circ\zeta_i\Big)\\
&=\bigsqcup_{i\in I}\dphi\uphi(\mu'\circ\zeta_i)&\text{(Proposition \ref{la_condition}(iii))}\\
&=\bigsqcup_{i\in I}\dphi\uphi\zeta_i^*\mu'\\
&=\bigsqcup_{i\in I}\CM(\zeta_i,\eta_i)\mu'
\end{align*}
for all $\mu'\in\CM\psi$, completing the proof.
\end{proof}


If $F:\bbA\to\bbB$ is a left adjoint $\CQ$-functor with $G:\bbB\to\bbA$ being its right adjoint, then
$$(F,G):(\bbA:\bbA\oto\bbA)\to(\bbB:\bbB\oto\bbB)$$
is a Chu transform between identity $\CQ$-distributors.
It is easy to verify that the assignment $F\mapsto(F,G)$ defines a functor $\CI:\QCCat\to\QChu$, and the composite functor
$$\CJ:=(\QCCat\to^{\CI}\QChu\to^{(\Box^{\nat},\Box_{\nat})}\ChuCon(\QDist)^{\op})$$
is 2-functorial.

\begin{thm} \label{M_J_id}
 $\QCCat$ is a retract of $\ChuCon(\QDist)^{\op}$ (in $2\text{-}\CAT$).
\end{thm}

\begin{proof} It suffices to prove that the 2-functor $\CM\CJ$ is naturally isomorphic to the identity 2-functor on $\QCCat$.

First note that for each skeletal complete $\CQ$-category $\bbA$,
$$\CM\CJ\bbA=\{\mu\in\PA\mid\bbA^{\da}\bbA_{\ua}\mu=\mu\}=\{\sY_{\bbA} x\mid x\in\bbA_0\}.$$
Indeed, for each $\mu\in\PA$, from $\bbA(\sup\mu,-)=\PA(\mu,\sY_{\bbA}-)=\bbA\lda\mu$ one has
$$\bbA^{\da}\bbA_{\ua}\mu=(\bbA\lda\mu)\rda\bbA=\bbA(\sup\mu,-)\rda\bbA=\bbA(-,\sup\mu)=\sY_{\bbA}\sup\mu.$$
Thus $\CM\CJ\bbA\subseteq\{\sY_{\bbA} x\mid x\in\bbA_0\}$, and the reverse inclusion is easy.

By Yoneda lemma, the correspondence $x\mapsto\sY_{\bbA}x$ induces a fully faithful $\CQ$-functor
$$\sY_{\bbA}:\bbA\to\{\sY_{\bbA} x\mid x\in\bbA_0\}=\CM\CJ\bbA.$$
It is clear that $\sY_{\bbA}$ is surjective, hence an isomorphism of skeletal $\CQ$-categories.

To see the naturality of $\{\sY_{\bbA}\}$, for each left $\CQ$-functor $F:\bbA\to\bbB$ between skeletal complete $\CQ$-categories, we prove the commutativity of the following diagram:
$$\bfig
\square<700,500>[\bbA`\CM\CJ\bbA`\bbB`\CM\CJ\bbB;\sY_{\bbA}`F`\CM\CJ F=\CM(F^{\nat},G_{\nat})`\sY_{\bbB}]
\efig$$
This is easy since
\begin{align*}
\CM(F^{\nat},G_{\nat})\sY_{\bbA}x&=\bbB^{\da}\bbB_{\ua}(F^{\nat})^*\sY_{\bbA}x\\
&=\bbB^{\da}\bbB_{\ua}(F^{\nat}(-,x))&\text{(Proposition \ref{tphi_hphi_Yoneda})}\\
&=(\bbB\lda\bbB(-,Fx))\rda\bbB\\
&=\bbB(-,Fx)\\
&=\sY_{\bbB} Fx
\end{align*}
for all $x\in\bbA_0$.
\end{proof}

The universal property of the quotient quantaloid $\BB(\QDist)$ along with the following Lemma \ref{equiv_Chu_con_condition} ensures that $\CM$ factors uniquely through the quotient homomorphism $\si$ via a quantaloid homomorphism $\CM_{\rb}$:
$$\bfig
\qtriangle/->`->`-->/<1500,500>[\ChuCon(\QDist)^{\op} `\BB(\QDist)^{\op}`\QCCat;\si^{\op}`\CM`\CM_{\rb}]
\efig$$

\begin{lem} \label{equiv_Chu_con_condition}
For Chu connections $(\zeta,\eta),(\zeta',\eta'):(\phi:\bbA\oto\bbB)\to(\psi:\bbA'\oto\bbB')$, the following statements are equivalent:
\begin{itemize}
\item[\rm (i)] $(\zeta,\eta)\sim(\zeta',\eta')$.
\item[\rm (ii)] $\uphi\zeta^*=\uphi\zeta'^*:\PA'\to\PB$.
\item[\rm (iii)] $\dphi\uphi\zeta^*=\dphi\uphi\zeta'^*:\PA'\to\PA$.
\item[\rm (iv)] $\CM(\zeta,\eta)=\CM(\zeta',\eta'):\CM\psi\to\CM\phi$.
\end{itemize}
\end{lem}

\begin{proof}
(i)${}\Lra{}$(ii): Suppose $\phi\lda\zeta=\phi\lda\zeta'$. For each $\mu'\in\PA'$,
$$\uphi\zeta^*\mu'=\phi\lda(\mu'\circ\zeta)=(\phi\lda\zeta)\lda\mu'=(\phi\lda\zeta')\lda\mu'=\phi\lda(\mu'\circ\zeta')=\uphi\zeta'^*\mu'.$$

(ii)${}\Lra{}$(i): For each $x'\in\bbA'_0$, by Proposition \ref{tphi_hphi_Yoneda} one has
$$\phi\lda\zeta(-,x')=\phi\lda(\zeta^*\sY_{\bbA'}x')=\uphi\zeta^*\sY_{\bbA'}x'=\uphi\zeta'^*\sY_{\bbA'}x'=\phi\lda(\zeta'^*\sY_{\bbA'}x')=\phi\lda\zeta'(-,x').$$

(ii)${}\iff{}$(iii) and (iii)${}\Lra{}$(iv): Trivial.

(iv)${}\Lra{}$(iii): For any Chu connection $(\zeta,\eta):(\phi:\bbA\oto\bbB)\to(\psi:\bbA'\oto\bbB')$, from Lemma \ref{zeta_continuous} one derives
$$\dphi\uphi\zeta^*\dpsi\upsi\leq\dphi\uphi\dphi\uphi\zeta^*=\dphi\uphi\zeta^*:\PA'\to\PA,$$
hence $\dphi\uphi\zeta^*\dpsi\upsi=\dphi\uphi\zeta^*$, since the reverse inequality is trivial. Therefore
$$\dphi\uphi\zeta^*\mu'=\dphi\uphi\zeta^*\dpsi\upsi\mu'=\CM(\zeta,\eta)\dpsi\upsi\mu'=\CM(\zeta',\eta')\dpsi\upsi\mu'=\dphi\uphi\zeta'^*\dpsi\upsi\mu'=\dphi\uphi\zeta'^*\mu'$$
for all $\mu'\in\PA'$, showing that $\dphi\uphi\zeta^*=\dphi\uphi\zeta'^*$.
\end{proof}

A little surprisingly, $\CM_{\rb}$ turns out to be an equivalence of quantaloids:

\begin{thm} \label{main}
$\CM_{\rb}:\BB(\QDist)^{\op}\to\QCCat$ is an equivalence of quantaloids; hence, $\BB(\QDist)$ and $\QCCat$ are dually equivalent quantaloids.
\end{thm}

\begin{proof}
It suffices to check that  $\CM_{\rb}$ is fully faithful and essentially surjective on objects. First, the definition of $\CM_{\rb}$ guarantees its fullness and faithfulness by Proposition \ref{M_functor} and Lemma \ref{equiv_Chu_con_condition}. Second, let $\CJ_{\rb}$ be the composite 2-functor
$$\QCCat\to^{\CJ}\ChuCon(\QDist)^{\op} \to^{\si^{\op}}\BB(\QDist)^{\op}.$$
Then
$$\CM_{\rb}\CJ_{\rb}= \CM_{\rb}\si^{\op} \CJ=\CM\CJ,$$
showing that $\CM_{\rb}\CJ_{\rb}$ is naturally isomorphic to the identity 2-functor on $\QCCat$ and, in particular, $\CM_{\rb}$ is essentially surjective on objects. Therefore, $\CM_{\rb}$ and $\CJ_{\rb}$ are both equivalences of quantaloids.
\end{proof}


\subsection{Chu correspondences} \label{ChuCor}

A \emph{$\CQ$-typed set} $\sA$ consists of a set $\sA_0$ and a type map $t:\sA_0\to\ob\CQ$. The category of $\CQ$-typed sets and type-preserving maps is exactly the slice category $\Set/\ob\CQ$. Each $\CQ$-typed set $\sA$ may be viewed as a \emph{discrete} $\CQ$-category with
$$\sA(x,y)=\begin{cases}
1_{tx}, &\text{if}\ x=y;\\
\bot_{tx,ty}, &\text{otherwise}.
\end{cases}$$
Type-preserving maps then become $\CQ$-functors between discrete $\CQ$-categories, making $\Set/\ob\CQ$ a full coreflective subcategory of $\QCat$, with the coreflector $|\text{-}|:\QCat\to\Set/\ob\CQ$ sending each $\CQ$-category $\bbA$ to its underlying $\CQ$-typed set $|\bbA|$.

A \emph{$\CQ$-matrix} (also \emph{$\CQ$-relation}) \cite{Betti1983,Heymans2010} $\phi:\sA\oto\sB$ between $\CQ$-typed sets is exactly a $\CQ$-distributor between discrete $\CQ$-categories. The category $\QMat$ of $\CQ$-typed sets and $\CQ$-matrices is clearly a full subquantaloid of $\QDist$, with $\sA:\sA\oto\sA$ playing as the identity $\CQ$-matrix on each $\CQ$-typed set $\sA$. 

\begin{prop} \label{Q_rel_dist}
For $\CQ$-categories $\bbA$, $\bbB$ and $\CQ$-matrix $\phi:|\bbA|\oto|\bbB|$, the following statements are equivalent:
\begin{itemize}
\item[\rm (i)] $\phi:\bbA\oto\bbB$ is a $\CQ$-distributor.
\item[\rm (ii)] $\phi\circ\bbA\leq\phi$ and $\bbB\circ\phi\leq\phi$.
\item[\rm (iii)] $\bbA\leq\phi\rda\phi$ and $\bbB\leq\phi\lda\phi$.
\end{itemize}
\end{prop}

In the case that $\CQ$ is the two-element Boolean algebra ${\bf 2}$, $\QMat$ is exactly the quantaloid $\Rel$ of sets and binary relations (see Example \ref{ChuConRel}). In formal concept analysis (see the next section for more), a \emph{formal context} is a triple $(A,B,R)$, where $A,B$ are sets and $R\subseteq A\times B$ is a relation. \emph{Chu correspondences} between formal contexts, first introduced by Mori \cite{Mori2008}, are essentially \emph{closed} Chu connections (defined above Proposition \ref{siuv}) in the quantaloid $\Rel$, and thus can be extended to general $\CQ$-matrices:

\begin{defn}[Mori \cite{Mori2008} for the case $\CQ={\bf 2}$]
A \emph{Chu correspondence} is a closed Chu connection $(\zeta,\eta):(\phi:\sA\oto\sB)\to(\psi:\sA'\oto\sB')$ between $\CQ$-matrices.
\end{defn}

Careful readers may have noticed that our definition of Chu correspondences here deviates a little bit from the original \cite[Definition 2]{Mori2008} for the case $\CQ={\bf 2}$, where $\eta$ is required to be a $\CQ$-matrix from $\sB'$ to $\sB$. In fact, in the case $\CQ={\bf 2}$, the dual of a $\CQ$-matrix from $\sB'$ to $\sB$ is a $\CQ$-matrix from $\sB$ to $\sB'$ (since $\CQ=\CQ^{\op}$), thus the direction of a $\CQ$-matrix is not important. So, our definition of Chu correspondences is essentially the same as that of Mori in the case $\CQ={\bf 2}$. It is in the general setting that  the direction of the involved $\CQ$-matrices matters.

The category of $\CQ$-matrices and Chu correspondences is, by definition, the quantaloid $\BB(\QMat)$, which is a full subquantaloid of $\BB(\QDist)$.

Denoting by $|\phi|:|\bbA|\oto|\bbB|$ for the underlying $\CQ$-matrix of a $\CQ$-distributor $\phi:\bbA\oto\bbB$, we point out an important fact of $\CM\phi$:

\begin{lem} \label{Mphi_invariant}
$\CM\phi=\CM|\phi|$.
\end{lem}

\begin{proof}
It suffices to show that $\mu\in\CM|\phi|$ implies $\mu\in\PA$ for any $\mu\in\CP|\bbA|$. Indeed,
\begin{align*}
\mu\circ\bbA&=((|\phi|\lda\mu)\rda|\phi|)\circ\bbA&(\mu\in\CM|\phi|)\\
&\leq((|\phi|\lda\mu)\rda|\phi|)\circ(|\phi|\rda|\phi|)&(\text{Proposition \ref{Q_rel_dist}(iii)})\\
&\leq(|\phi|\lda\mu)\rda|\phi|\\
&=\mu.&(\mu\in\CM|\phi|)
\end{align*}
Thus by Proposition \ref{Q_rel_dist}(ii) one has $\mu\in\PA$.
\end{proof}

The above lemma shows that $\CM\phi$ is independent of the $\CQ$-categorical structures of the domain and codomain of $\phi$.

\begin{thm}\label{main result 3}
$\BB(\QMat)$ and $\QCCat$ are dually equivalent quantaloids. Thus one has equivalences of quantaloids
$$\BB(\QDist)^{\op}\simeq\BB(\QMat)^{\op}\simeq\QCCat.$$
\end{thm}

\begin{proof}
Since $\QMat$ is a full subquantaloid of $\QDist$, it follows from the definition of back diagonals that $\BB(\QMat)$ is a full subquantaloid of $\BB(\QDist)$. The above lemma ensures that the composite of $\CM_{\rb}$ with the inclusion $\BB(\QMat)^{\op}\ \to/^(->/\BB(\QDist)^{\op}$ is fully faithful and essentially surjective on objects, hence $\BB(\QMat)$ and $\QCCat$ are dually equivalent.
\end{proof}

In the case $\CQ={\bf 2}$, since $\Sup$ (=${\bf 2}\text{-}\CCat$) is self-dual, it follows that $\BB({\bf 2}\text{-}\Mat)=\BB(\Rel)$ itself is equivalent to $\Sup$. This is the content of the main result in \cite{Mori2008}:

\begin{cor}[Mori \cite{Mori2008}]
The category  of formal contexts and Chu correspondences is equivalent to the category   of complete lattices and join-preserving maps.
\end{cor}

\subsection{Dualization}

The isomorphism
$$(-)^{\op}:\QDist\to(\CQ^{\op}\text{-}\Dist)^{\op}$$
in Remark \ref{Qcat_dual} induces an isomorphism of quantaloids
\begin{equation} \label{ChuCon_QDist_op}
(-)^{\op}:\ChuCon(\QDist)\to\ChuCon(\CQ^{\op}\text{-}\Dist)^{\op}
\end{equation}
that sends a Chu connection $(\zeta,\eta):\phi\to\psi$ to its dual $(\eta^{\op},\zeta^{\op}):\psi^{\op}\to\phi^{\op}$.
\begin{equation} \label{zeta_eta_op}
\bfig
\square<600,500>[\bbA`\bbA'`\bbB`\bbB';\zeta`\phi`\psi`\eta]
\place(0,250)[\circ] \place(600,250)[\circ] \place(300,0)[\circ] \place(300,500)[\circ]
\place(1050,250)[\mapsto]
\square(1500,0)<700,500>[(\bbB')^{\op}`\bbB^{\op}`(\bbA')^{\op}`\bbA^{\op};\eta^{\op}`\psi^{\op}`\phi^{\op}`\zeta^{\op}]
\place(1500,250)[\circ] \place(2200,250)[\circ] \place(1850,0)[\circ] \place(1850,500)[\circ]
\efig
\end{equation}

The functor $\CM:\ChuCon(\QDist)^{\op}\to\QCCat$ preserves the dualization of Chu connections up to a natural isomorphism as shown below.

\begin{prop} \label{M_preserve_dual}
For each $\CQ$-distributor $\phi:\bbA\oto\bbB$, $\CM\phi^{\op}$ is isomorphic to $(\CM\phi)^{\op}$. Furthermore, the diagram
$$\bfig
\square<1500,500>[\ChuCon(\QDist)^{\op}`\ChuCon(\CQ^{\op}\text{-}\Dist)`\QCCat`(\CQ^{\op}\text{-}\CCat)^{\op};(-)^{\op}`\CM`\CM^{\op}`\dv{}^{\op}]
\efig$$
commutes up to a natural isomorphism, where $\dv{}^{\op}:\QCCat\to(\CQ^{\op}\text{-}\CCat)^{\op}$ is the isomorphism given in Proposition \ref{QCCat_QopCCatop_iso}.
\end{prop}

\begin{proof}
First, it is not difficult to verify that
$$\uphi:\CM\phi\to\Fix(\uphi\dphi)$$
is an isomorphism of $\CQ$-categories (with $\dphi:\Fix(\uphi\dphi)\to\CM\phi$ as its inverse), and so is
$$(-)^{\op}:\Fix(\uphi\dphi)\to(\CM\phi^{\op})^{\op},\quad(\lam:*_X\oto\bbB)\mapsto(\lam^{\op}:\bbB^{\op}\oto *_X).$$
Thus one soon has the isomorphism of $\CM\phi^{\op}$ and $(\CM\phi)^{\op}$, which are respectively the images of $\phi$ under $\CM^{\op}\cdot(-)^{\op}$ and $\dv{}^{\op}\cdot\CM$.

Second, we show that $\al_{\phi}:=(-)^{\op}\cdot\uphi^{\op}$ gives rise to a natural isomorphism $\al$ from $\dv{}^{\op}\cdot\CM$ to $\CM^{\op}\cdot(-)^{\op}$. For the naturality we must check the commutativity of the diagram
\begin{equation} \label{dualization_natural}
\bfig
\square/->`->``->/<1200,500>[(\CM\phi)^{\op}`\Fix(\uphi\dphi)^{\op}`(\CM\psi)^{\op}`\Fix(\upsi\dpsi)^{\op};
\uphi^{\op}`\dv{}^{\op}\cdot\CM(\zeta,\eta)``\upsi^{\op}]
\square(1200,0)/->``->`->/<1200,500>[\Fix(\uphi\dphi)^{\op}`\CM\phi^{\op}`\Fix(\upsi\dpsi)^{\op}`\CM\psi^{\op};
(-)^{\op}``\CM(\eta^{\op},\zeta^{\op})`(-)^{\op}]
\efig
\end{equation}
for all Chu connections $(\zeta,\eta):(\phi:\bbA\oto\bbB)\to(\psi:\bbA'\oto\bbB')$. From Step 1 in the proof of Proposition \ref{M_functor} one already knows $\CM(\zeta,\eta)\dv\zeta_*:\CM\phi\to\CM\psi$, thus
$$\dv{}^{\op}\cdot\CM(\zeta,\eta)=\zeta_*^{\op}:(\CM\phi)^{\op}\to(\CM\psi)^{\op}.$$
Note also that for all $\mu\in\CM\phi$,
\begin{align*}
\upsi\zeta_*\mu&=\psi\lda(\mu\lda\zeta)\\
&=\psi\lda(((\phi\lda\mu)\rda\phi)\lda\zeta)&(\mu\in\CM\phi)\\
&=\psi\lda((\phi\lda\mu)\rda(\phi\lda\zeta))\\
&=\psi\lda((\phi\lda\mu)\rda(\eta\rda\psi))&((\zeta,\eta)\ \text{is a Chu connection})\\
&=\psi\lda((\eta\circ(\phi\lda\mu))\rda\psi))\\
&=\upsi\dpsi\eta^{\dag}\uphi\mu.
\end{align*}
Hence
\begin{align*}
\CM(\eta^{\op},\zeta^{\op})(\uphi\mu)^{\op}&=(\psi^{\op})^{\da}(\psi^{\op})_{\ua}(\eta^{\op})^*(\uphi\mu)^{\op}\\
&=(\psi^{\op}\lda((\uphi\mu)^{\op}\circ\eta^{\op}))\rda\psi^{\op}\\
&=(\psi\lda((\eta\circ\uphi\mu)\rda\psi))^{\op}\\
&=(\upsi\dpsi\eta^{\dag}\uphi\mu)^{\op}\\
&=(\upsi\zeta_*\mu)^{\op},
\end{align*}
indicating the commutativity of the diagram (\ref{dualization_natural}).
\end{proof}

It is clear that the image $(\eta^{\op},\zeta^{\op}):\psi^{\op}\to\phi^{\op}$ under the assignment (\ref{zeta_eta_op}) is a closed Chu connection whenever so is $(\zeta,\eta):\phi\to\psi$, thus it also induces an isomorphism of quantaloids
$$(-)^{\op}:\BB(\QDist)\to\BB(\CQ^{\op}\text{-}\Dist)^{\op}.$$
Similarly, the functor $\CM_{\rb}:\BB(\QDist)^{\op}\to\QCCat$ also preserves the dualization of back diagonals up to a natural isomorphism, and we do not bother spelling it out here.

\section{Reduction of $\CQ$-distributors} \label{Reduction}

Formal concept analysis \cite{Davey2002,Ganter1999} is an important tool in data analysis. A relation $R\subseteq A\times B$ between sets is called a \emph{formal context} in this theory and usually written as a triple $(A,B,R)$, with $A$ interpreted as the set of \emph{objects}, $B$ the set of \emph{properties}, and $(x,y)\in R$ reads as the object $x$ has property $y$. The Galois connection
$$\uR\dv\dR:({\bf 2}^B)^{\op}\to{\bf 2}^A$$
presented in Example \ref{ChuConRel} plays a fundamental role in formal concept analysis. A pair $(U,V)\in{\bf 2}^A\times({\bf 2}^B)^{\op}$ is called a \emph{formal concept} if $U=\dR(V)$ and $V=\uR(U)$. Formal concepts of a formal context $(A,B,R)$ constitute a complete lattice with the order
$$(U_1,V_1)\leq(U_2,V_2)\iff U_1\subseteq U_2\quad\text{and}\quad V_2\subseteq V_1,$$
called the \emph{concept lattice} of the formal context $(A,B,R)$, which is isomorphic to $\CM R$ with $R$ considered as a ${\bf 2}$-distributor between sets equipped with the discrete order. In particular, if $R$ is a partial order on a set $A$, then $\CM R$ is the Dedekind-MacNeille completion of the partially ordered set $(A,R)$.

An important problem in the application of formal concept analysis is the \emph{reduction} of formal contexts. While dealing with a large quantity of data, one always wants to reduce the size of the set of objects and/or that of properties without affecting the structure of the concept lattice. Intuitively, given a formal context $(A,B,R)$, one wishes to find subsets $A'$, $B'$ of $A$, $B$, respectively, such that $\CM R_{A',B'}$ is isomorphic to $\CM R$, where $R_{A',B'}$ is the restriction of $R$ to $A'\times B'$. A closer look reveals that this intuition needs clarification. To see this, let $\bQ$ denote the set of rational numbers, consider the partially ordered sets $(\bQ,\leq)$ and $(\bQ\cap[0,1],\leq)$ (identified with the formal contexts $(\bQ,\bQ,\leq)$  and  $(\bQ\cap[0,1],\bQ\cap[0,1],\leq)$, respectively). The Dedekind-MacNeille completion of them are isomorphic (as lattices), but  it is counter-intuitive that $(\bQ\cap[0,1],\bQ\cap[0,1],\leq)$ is a reduct of $(\bQ,\bQ,\leq)$: too much information in  $(\bQ,\bQ,\leq)$ has been thrown away. So, a right step to a theory of reduction of formal contexts is to require that $\CM R$ and $\CM R_{A',B'}$ are not only isomorphic, but also isomorphic in a  canonical  way. In this section, we will employ Chu connections to establish a rigorous theory of reduction that is compatible with this intuition.

Note that for a small quantaloid $\CQ$, a $\CQ$-distributor $\phi:\bbA\oto\bbB$ between $\CQ$-categories may be thought of as a multi-typed and multi-valued relation that respects $\CQ$-categorical structures in its domain and codomain. Consequently, the induced Isbell adjunction
$$\uphi\dv\dphi:\PdB\to\PA$$
and its image $\CM\phi$ under $\CM:\ChuCon(\QDist)^{\op}\to\QCCat$ present a categorical version of formal concept analysis. Therefore, the theory of reduction of formal contexts will be established as a generalized version here, i.e., a theory of reduction of $\CQ$-distributors.

\subsection{Comparison $\CQ$-functors and reducts}

Before proceeding, we fix some notations. Given a $\CQ$-category $\bbA$, $\bbA'\subseteq\bbA$ indicates that $\bbA'$ is a $\CQ$-subcategory of $\bbA$, with hom-arrows inherited from $\bbA$. Correspondingly, $\bbA\setminus\bbA'$ denotes the complementary $\CQ$-subcategory of $\bbA'$ in $\bbA$. For a $\CQ$-distributor $\phi:\bbA\oto\bbB$ and $\bbA'\subseteq\bbA$, $\bbB'\subseteq \bbB$, we always write $I:\bbA'\to\bbA$ and $J:\bbB'\to\bbB$ for the inclusion $\CQ$-functors, and  $\phi_{\bbA',\bbB'}:\bbA'\oto\bbB'$ for the restriction of $\phi$ on $\bbA'$ and $\bbB'$.
 In particular, we  write $\mu_{\bbA'}$ (resp. $\lam_{\bbA'}$) for the restriction of $\mu:\bbA\oto *_{t\mu}$ (resp. $\lam:*_{t\lam}\oto\bbA$) on $\bbA'$.

For each $\CQ$-distributor $\phi:\bbA\oto\bbB$ and $\bbA'\subseteq\bbA$, $\bbB'\subseteq\bbB$,
it is easy to see that $(\bbA, J_{\nat}): \phi_{\bbA,\bbB'}\to\phi$ and $(I^{\nat},\bbB):\phi\to\phi_{\bbA',\bbB}$
are both Chu connections.
$$\bfig
\square<600,500>[\bbA`\bbA`\bbB'`\bbB; \bbA`\phi_{\bbA,\bbB'}`\phi`J_{\nat}]
\morphism(600,500)|b|<-600,-500>[\bbA`\bbB';]
\place(0,250)[\circ] \place(600,250)[\circ] \place(300,500)[\circ] \place(300,0)[\circ]\place(300,250)[\circ]
\square(1500,0)<600,500>[\bbA`\bbA'`\bbB`\bbB;I^{\nat}`\phi`\phi_{\bbA',\bbB}`\bbB]
\morphism(2100,500)|b|<-600,-500>[\bbA'`\bbB;]
\place(1500,250)[\circ] \place(2100,250)[\circ] \place(1800,500)[\circ] \place(1800,0)[\circ]\place(1800,250)[\circ]
\efig$$

Hence, we obtain  two left adjoint $\CQ$-functors:
\begin{align}
&\CM(\bbA,J_{\nat}):\CM\phi\to\CM\phi_{\bbA,\bbB'},\label{M_A_JBp}\\
&\CM(I^{\nat},\bbB):\CM\phi_{\bbA',\bbB}\to\CM\phi. \label{M_JAp_B}
\end{align}
Replacing $\bbA$ by $\bbA'$ in (\ref{M_A_JBp}) and $\bbB$ by $\bbB'$ in (\ref{M_JAp_B}) we obtain another two left adjoint $\CQ$-functors:
\begin{align*}
&\CM(\bbA', J_{\nat}): \CM\phi_{\bbA',\bbB}\to\CM\phi_{\bbA',\bbB'}, \\
&\CM(I^{\nat},\bbB'): \CM\phi_{\bbA',\bbB'}\to\CM\phi_{\bbA,\bbB'}.
\end{align*}

Composing the right adjoint
of $\CM(I^{\nat},\bbB)$ with $\CM(\bbA', J_{\nat})$ gives a $\CQ$-functor
$$R_1: \CM\phi\to\CM\phi_{\bbA',\bbB} \to\CM\phi_{\bbA',\bbB'};$$
composing $\CM(\bbA, J_{\nat})$ with the right adjoint of $\CM(I^{\nat},\bbB')$ gives a $\CQ$-functor
$$R_2:\CM\phi\to\CM\phi_{\bbA,\bbB'} \to \CM\phi_{\bbA',\bbB'};$$
composing the right adjoint of $\CM(\bbA', J_{\nat})$ with   $\CM(I^{\nat},\bbB)$ gives a $\CQ$-functor
$$E_1:\CM\phi_{\bbA',\bbB'}\to\CM\phi_{\bbA',\bbB} \to\CM\phi;$$
and finally, composing $\CM(I^{\nat},\bbB')$ with the right adjoint of $\CM(\bbA, J_{\nat})$ gives a $\CQ$-functor
$$E_2:\CM\phi_{\bbA',\bbB'}\to\CM\phi_{\bbA,\bbB'} \to\CM\phi.$$

The four $\CQ$-functors $R_1$, $R_2$, $E_1$, $E_2$ arise in a natural way, so, they can be employed to play the role of ``comparison $\CQ$-functors'' between $\CM\phi$ and $\CM\phi_{\bbA',\bbB'}$. The following conclusion is of crucial importance in this regard.

\begin{thm}\label{one for all}
Given a $\CQ$-distributor $\phi:\bbA\oto\bbB$   and $\bbA'\subseteq\bbA, \bbB'\subseteq \bbB$, if one of the $\CQ$-functors  $R_1$, $R_2$, $E_1$, $E_2$ is an isomorphism in $\QCat$, then so are the other three. 
\end{thm}

This theorem leads to the core definition of this section:

\begin{defn}
Given a $\CQ$-distributor $\phi:\bbA\oto\bbB$   and $\bbA'\subseteq\bbA$, $\bbB'\subseteq\bbB$, we say that the restriction $\phi_{\bbA',\bbB'}$ is a reduct of $\phi$ if one (hence each) of the $\CQ$-functors  $R_1$, $R_2$, $E_1$, $E_2$ is an isomorphism in $\QCat$.
\end{defn}

The aim of this subsection is to prove Theorem \ref{one for all}; the next subsection will present a sufficient and necessary condition for $\phi_{\bbA',\bbB'}$ to be a reduct of $\phi$.

 From now on throughout this section, $\phi$ is always assumed to be a $\CQ$-distributor $\bbA\oto\bbB$, while $\bbA'$, $\bbB'$ are $\CQ$-subcategories of $\bbA$, $\bbB$, with $I:\bbA'\to\bbA$, $J:\bbB'\to\bbB$ being the inclusion $\CQ$-functors, respectively.


The restriction map
$$(-)_{\bbA',\bbB'}:\QDist(\bbA,\bbB)\to\QDist(\bbA',\bbB')$$
has both left and right adjoints 
$$\underline{(-)}\dv(-)_{\bbA',\bbB'}\dv\overline{(-)},$$
which extend a $\CQ$-distributor $\phi':\bbA'\oto\bbB'$ respectively to $\underline{\phi'}:\bbA\oto\bbB$ and $\overline{\phi'}:\bbA\oto\bbB$ with
\begin{equation} \label{underline_overline_phi_def}
\underline{\phi'}= J_{\nat}\circ\phi'\circ I^{\nat}\quad\text{and}\quad\overline{\phi'}=(J^{\nat}\rda\phi')\lda I_{\nat}=J^{\nat}\rda(\phi'\lda I_{\nat}).
\end{equation}


The verification of the following proposition is easy under the help of Propositions \ref{adjoint_arrow_calculation} and \ref{fully_faithful_graph}:

\begin{prop} \phantomsection \label{subdistributor_inclusion}
\begin{itemize}
\item[\rm (1)] $\phi_{\bbA',\bbB'}=J^{\nat}\circ\phi\circ I_{\nat}=( J_{\nat}\rda\phi)\lda I^{\nat}= J_{\nat}\rda(\phi\lda I^{\nat})$.
\item[\rm (2)] $\phi_{\bbA,\bbB'}\lda\mu=(\phi\lda\mu)_{\bbB'}$ and $\lam\rda\phi_{\bbA',\bbB}=(\lam\rda\phi)_{\bbA'}$ for all $\mu\in\PA$, $\lam\in\PdB$.
\item[\rm (3)] $\phi_{\bbA',\bbB}\lda\mu'=\phi\lda\underline{\mu'}$ and $\lam'\rda\phi_{\bbA,\bbB'}=\underline{\lam'}\rda\phi$ for all $\mu'\in\PA'$, $\lam'\in\PdB'$.
\item[\rm (4)] $\mu'\lda\phi_{\bbA',\bbB}=\overline{\mu'}\lda\phi$ and $\phi_{\bbA,\bbB'}\rda\lam'=\phi\rda\overline{\lam'}$ for all $\mu'\in\PA'$, $\lam'\in\PdB'$.
\end{itemize}
\end{prop}







These formulas will be used in a flexible way throughout this section.  In particular, the first formula indicates that the map
$$(-)_{\bbA'}:\PA\to\PA'$$
sending each $\mu:\bbA\oto *_{t\mu}$ to its restriction $\mu_{\bbA'}:\bbA'\oto *_{t\mu}$ is in fact the $\CQ$-functor
$$ I_{\nat}^*=(I^{\nat})_*:\PA\to\PA'.$$


\begin{lem} \phantomsection \label{A_JB_Chu_connection_M}
\begin{itemize}
\item[\rm (1)] $\CM\phi_{\bbA,\bbB'}\subseteq\CM\phi$, and the inclusion $\CQ$-functor $\CM\phi_{\bbA,\bbB'}\ \to/^(->/\CM\phi$ is right adjoint to $\CM(\bbA, J_{\nat}):\CM\phi\to\CM\phi_{\bbA, \bbB'}$. In particular, $\CM(\bbA, J_{\nat})$ is surjective.
\item[\rm (2)] $\CM(I^{\nat},\bbB): \CM\phi_{\bbA',\bbB}\to\CM\phi$ is fully faithful, with a right adjoint given by $(-)_{\bbA'}: \CM\phi\to\CM\phi_{\bbA',\bbB}$.\footnote{This implies, in particular,  that for each $\mu\in\CM\phi\subseteq\PA$, the restriction $\mu_{\bbA'}$ on $\bbA'$ belongs to $\CM\phi_{\bbA',\bbB}$. }
\end{itemize}
\end{lem}

\begin{proof}
(1) It follows from the proof of Proposition \ref{M_functor}, Step 1, that the right adjoint of $\CM(\bbA, J_{\nat})$ is given by
$$\bbA_*:\CM\phi_{\bbA,\bbB'}\to\CM\phi, \quad \bbA_*(\mu)=\mu\lda\bbA=\mu;$$
hence, the conclusion follows.

(2) First, by virtue of Proposition \ref{M_functor}, the right adjoint of $\CM(I^{\nat},\bbB)$ is given by $(I^{\nat})_*$, hence by $(-)_{\bbA'}:\CM\phi\to\CM\phi_{\bbA',\bbB}$.
Second, for all $\mu'\in\CM\phi_{\bbA',\bbB}$, from Equation (\ref{underline_overline_phi_def}) one has
$$\CM(I^{\nat},\bbB)\mu'=\dphi\uphi(I^{\nat})^*\mu'=\dphi\uphi(\mu'\circ I^{\nat})=\dphi\uphi\underline{\mu'}.$$
Therefore, for all $\mu',\mu''\in\CM\phi_{\bbA',\bbB}$,
\begin{align*}
\PA'(\mu',\mu'')&=\PA'(\mu',\dphiApB\uphiApB\mu'')&(\mu''\in\CM\phi_{\bbA',\bbB})\\
&=\PdB(\uphiApB\mu',\uphiApB\mu'')&(\uphiApB\dv\dphiApB)\\
&=\PdB(\uphi\underline{\mu'},\uphi\underline{\mu''})&(\text{Proposition \ref{subdistributor_inclusion}(3)})\\
&=\PdB(\uphi\dphi\uphi\underline{\mu'},\uphi\underline{\mu''})\\
&=\PA(\dphi\uphi\underline{\mu'}, \dphi\uphi\underline{\mu''}),&(\uphi\dv\dphi)
\end{align*}
indicating that $\CM(I^{\nat},\bbB)$ is fully faithful.
\end{proof}

Therefore, the four comparison $\CQ$-functors become:
\renewcommand\arraystretch{2}
$$\begin{array}{ll}
R_1=(\CM\phi\to^{(-)_{\bbA'}}\CM\phi_{\bbA',\bbB} \to^{\CM(\bbA', J_{\nat})}\CM\phi_{\bbA',\bbB'}), & \mu\mapsto\dphiApBp\uphiApBp\mu_{\bbA'},\\
E_1=(\CM\phi_{\bbA',\bbB'}\ \to/^(->/\CM\phi_{\bbA',\bbB}\to^{\CM(I^{\nat},\bbB)}\CM\phi), & \mu'\mapsto\dphi\uphi\underline{\mu'},\\
R_2=(\CM\phi\to^{\CM(\bbA, J_{\nat})}\CM\phi_{\bbA,\bbB'}\to^{(-)_{\bbA'}}\CM\phi_{\bbA',\bbB'}), & \mu\mapsto(\dphiABp\uphiABp\mu)_{\bbA'},\\
E_2=(\CM\phi_{\bbA',\bbB'}\to^{\CM(I^{\nat},\bbB')}\CM\phi_{\bbA,\bbB'}\ \to/^(->/\CM\phi), & \mu'\mapsto\dphiABp\uphiABp\underline{\mu'}.
\end{array}$$
As an immediate consequence of Lemma \ref{A_JB_Chu_connection_M} we obtain that $R_1$ and $R_2$ are both surjective, while $E_1$ and $E_2$ are both fully faithful (thus injective, since their domains are skeletal).

\begin{lem} \phantomsection \label{RE_id}
\begin{itemize}
\item[\rm (1)] The diagram
$$\bfig
\square<1000,500>[\CM\phi_{\bbA',\bbB'}`\CM\phi`\CM\phi`\CM\phi_{\bbA',\bbB};E_1`E_2`(-)_{\bbA'}`(-)_{\bbA'}]
\morphism(0,500)/^(->/<1000,-500>[\CM\phi_{\bbA',\bbB'}`\CM\phi_{\bbA',\bbB};]
\efig$$
is commutative. \item[\rm (2)]   All of the composites $R_1 E_1$, $R_1 E_2$, $R_2 E_1$, $R_2 E_2$ coincide with the identity $\CQ$-functor on $\CM\phi_{\bbA',\bbB'}$. In particular, $\CM\phi_{\bbA',\bbB'}$ is a retract of $\CM\phi$ (in $\QCat$). \end{itemize}
\end{lem}

\begin{proof}
(1) For all $\mu'\in\CM\phi_{\bbA',\bbB}$, it holds that
\begin{equation} \label{dphi_uphi_mup_Ap}
(\dphi\uphi\underline{\mu'})_{\bbA'}=\dphiApB\uphi\underline{\mu'}=\dphiApB\uphiApB\mu'=\mu',
\end{equation}
where the first and second equalities respectively follow from items (2) and (3) in Proposition \ref{subdistributor_inclusion}.
Therefore, if $\mu'\in\CM\phi_{\bbA',\bbB'}$, then $\mu'\in\CM\phi_{\bbA',\bbB}$ (since $\CM\phi_{\bbA',\bbB'}\subseteq\CM\phi_{\bbA',\bbB}$), and consequently,
$$(E_1\mu')_{\bbA'}=(\dphi\uphi\underline{\mu'})_{\bbA'} =\mu'=(\dphiABp\uphiABp\underline{\mu'})_{\bbA'} =(E_2\mu')_{\bbA'},$$
where the third equality holds by applying Equation (\ref{dphi_uphi_mup_Ap}) to $\phi_{\bbA,\bbB'}$. This proves the commutativity of the diagram.

(2) For all $\mu'\in\CM\phi_{\bbA',\bbB'}$, note that
\begin{align*}
R_1E_1\mu'&=\dphiApBp\uphiApBp(\dphi\uphi\underline{\mu'})_{\bbA'}\\
&=\dphiApBp\uphiApBp\mu'&(\text{Equation (\ref{dphi_uphi_mup_Ap})})\\
&=\mu'&(\mu'\in\CM\phi_{\bbA',\bbB'})\\
&=(\dphiABp\uphiABp \underline{\mu'})_{\bbA'}&(\text{Equation (\ref{dphi_uphi_mup_Ap})})\\
&=(\dphiABp\uphiABp\dphiABp\uphiABp\underline{\mu'})_{\bbA'}\\
&=R_2E_2\mu'
\end{align*}
and
\begin{align*}
R_1 E_2\mu'&=\dphiApBp\uphiApBp(\dphiABp\uphiABp\underline{\mu'})_{\bbA'}\\
&=\dphiApBp\uphiApBp\mu'&(\text{Equation (\ref{dphi_uphi_mup_Ap})})\\
&=\mu'&(\mu'\in\CM\phi_{\bbA',\bbB'})\\
&=(\dphiABp\uphiABp\underline{\mu'})_{\bbA'}&(\text{Equation (\ref{dphi_uphi_mup_Ap})})\\
&=(\dphiABp(\uphi\underline{\mu'})_{\bbB'})_{\bbA'}&(\text{Proposition \ref{subdistributor_inclusion}(2)})\\
&=(\dphiABp(\uphi\dphi\uphi\underline{\mu'})_{\bbB'})_{\bbA'}\\
&=(\dphiABp\uphiABp\dphi\uphi\underline{\mu'})_{\bbA'}&(\text{Proposition \ref{subdistributor_inclusion}(2)})\\
&=R_2 E_1\mu',
\end{align*}
hence all of $R_1 E_1$, $R_1 E_2$, $R_2 E_1$, $R_2 E_2$ coincide with the identity $\CQ$-functor on $\CM\phi_{\bbA',\bbB'}$.
\end{proof}

The proof of Theorem \ref{one for all} then easily comes out of Lemma \ref{RE_id}:

\begin{proof}[Proof of Theorem \ref{one for all}]
Since all of $R_1 E_1$, $R_1 E_2$ and $R_2 E_2$ coincide with the identity $\CQ$-functor on $\CM\phi_{\bbA',\bbB'}$, it soon follows that $E_1$ is an isomorphism in $\QCat$ if and only if so is $R_1$ if and only if so is $E_2$ if and only if so is $R_2$.
\end{proof}

\begin{cor}
If one of the $\CQ$-functors  $R_1$, $R_2$, $E_1$, $E_2$ is an isomorphism in $\QCat$, then $R_1=R_2$, $E_1=E_2$. Moreover, for each $\mu\in\CM\phi\subseteq\PA$, $R_1\mu=R_2\mu=\mu_{\bbA'}$, the restriction of $\mu$ on $\bbA'$.
\end{cor}

\begin{proof}
That $R_1=R_2$ and $E_1=E_2$ follow immediately from that all of $R_1 E_1$, $R_2 E_2$, $R_1 E_2$ coincide with the identity $\CQ$-functor on $\CM\phi_{\bbA',\bbB'}$. So, it remains to check that $R_1\mu=\mu_{\bbA'}$ for all $\mu\in\CM\phi$. Note that $\CM(I^{\nat},\bbB)$ is surjective since so is $E_1$, hence an isomorphism in $\QCat$ because it is already fully faithful. Then the inclusion $\CQ$-functor $\CM\phi_{\bbA',\bbB'}\ \to/^(->/\CM\phi_{\bbA',\bbB}$ is surjective, hence  an identity $\CQ$-functor. Therefore, $\CM(\bbA',J_\natural)$, being left adjoint to an identity $\CQ$-functor, itself must be an identity $\CQ$-functor, so, $R_1\mu=\mu_{\bbA'}$ for all $\mu\in\CM\phi$.
\end{proof}

\subsection{Reducible $\CQ$-subcategories}

This subsection presents a characterization of reducts of $\CQ$-distributors in terms of reducible $\CQ$-subcategories.

\begin{defn} \label{phi_reducible_def}
Let $\phi:\bbA\oto\bbB$ be a $\CQ$-distributor and $\bbA'\subseteq\bbA$. $\bbA\setminus\bbA'$ is \emph{$\phi$-reducible} if for any $\mu\in\PA$, there exists $\mu'\in\PA'$ such that
$$\phi\lda\mu=\phi_{\bbA',\bbB}\lda\mu'.$$

Dually, for $\bbB'\subseteq\bbB$, $\bbB\setminus\bbB'$ is \emph{$\phi$-reducible} if $\bbB^{\op}\setminus\bbB'^{\op}$ is $\phi^{\op}$-reducible w.r.t. the $\CQ^{\op}$-distributor $\phi^{\op}:\bbB^{\op}\oto\bbA^{\op}$; or equivalently, for any $\lam\in\PdB$, there exists $\lam'\in\PdB'$ such that
$$\lam\rda\phi=\lam'\rda\phi_{\bbA,\bbB'}.$$
\end{defn}

\begin{rem} \label{FCA_reducible}
In formal concept analysis, given a formal context $(A,B,R)$, an object $x\in A$ is \emph{reducible} \cite{Ganter1999} if there exists a subset $U\subseteq A\setminus\{x\}$ with $\uR(\{x\})=\uR(U)$, and a property $y\in B$ is \emph{reducible} if there exists a subset $V\subseteq B\setminus\{y\}$ with $\dR(\{y\})=\dR(V)$. It is easy to see that for any subset $A'\subseteq A$ (resp. $B'\subseteq B$), $A\setminus A'$ (resp. $B\setminus B'$) is $R$-reducible in the sense of Definition \ref{phi_reducible_def} if, and only if, each element $x\in A\setminus A'$ (resp. $y\in B\setminus B'$) is reducible. Therefore, the $\phi$-reducibility introduced here is an extension of the classical notions in formal concept analysis.
\end{rem}

The main result of this subsection is the following:

\begin{thm} \label{reduct reducible}
$\phi_{\bbA',\bbB'}$ is a reduct of $\phi$ if and only if $\bbA\setminus\bbA'$ and $\bbB\setminus\bbB'$ are both $\phi$-reducible.
\end{thm}

As preparations for the proof of this theorem, we first present two propositions, which are special cases of the conclusion in Theorem \ref{reduct reducible}, and also justify the term ``$\phi$-reducible".

\begin{prop} \label{domain_reducible_condition}
Let $\phi:\bbA\oto\bbB$ be a $\CQ$-distributor and $\bbA'\subseteq\bbA$. The following statements are equivalent:
\begin{itemize}
\item[\rm (i)] $\bbA\setminus\bbA'$ is $\phi$-reducible.
\item[\rm (ii)] $\phi\lda\mu=\phi_{\bbA',\bbB}\lda(\dphi\uphi\mu)_{\bbA'}$ for all $\mu\in\PA$.
\item[\rm (iii)] $\uphi\dphi=\uphiApB\dphiApB:\PdB\to\PdB$.
\item[\rm (iv)] $\CM(I^{\nat},\bbB):\CM\phi_{\bbA',\bbB}\to\CM\phi$ is surjective, thus an isomorphism in $\QCat$.
\end{itemize}
\end{prop}

\begin{proof}
(i)${}\Lra{}$(ii): For each $\mu\in\PA$, there exists $\mu'\in\PA'$ with
$$\uphi\mu=\phi\lda\mu=\phi_{\bbA',\bbB}\lda\mu'$$
since $\bbA\setminus\bbA'$ is $\phi$-reducible. Therefore
\begin{align*}
\phi\lda\mu&=\phi_{\bbA',\bbB}\lda\mu'\\
&=\phi_{\bbA',\bbB}\lda((\phi_{\bbA',\bbB}\lda\mu')\rda\phi_{\bbA',\bbB})\\
&=\phi_{\bbA',\bbB}\lda((\uphi\mu)\rda\phi_{\bbA',\bbB})\\
&=\phi_{\bbA',\bbB}\lda((\uphi\mu)\rda\phi)_{\bbA'}&(\text{Proposition \ref{subdistributor_inclusion}(2)})\\
&=\phi_{\bbA',\bbB}\lda(\dphi\uphi\mu)_{\bbA'}.
\end{align*}

(ii)${}\Lra{}$(iii): For all $\lam\in\PdB$,
$$\uphi\dphi\lam=\phi\lda\dphi\lam=\phi_{\bbA',\bbB}\lda(\dphi\uphi\dphi\lam)_{\bbA'}=\phi_{\bbA',\bbB}\lda(\lam\rda\phi)_{\bbA'}=\uphiApB\dphiApB\lam,$$
where the last equality follows from Proposition \ref{subdistributor_inclusion}(2).

(iii)${}\Lra{}$(iv): With Lemma \ref{A_JB_Chu_connection_M}(2) at hand, it suffices to show $\mu=\dphi\uphi\underline{\mu_{\bbA'}}$ for all $\mu\in\CM\phi$. Indeed,
\begin{align*}
\mu&=\dphi\uphi\dphi\uphi\mu\\
&=\dphi\uphiApB\dphiApB\uphi\mu&(\text{iii})\\
&=\dphi\uphiApB((\dphi\uphi\mu)_{\bbA'})&(\text{Proposition \ref{subdistributor_inclusion}(2)})\\
&=\dphi\uphiApB\mu_{\bbA'}&(\mu\in\CM\phi)\\
&=\dphi\uphi\underline{\mu_{\bbA'}}.&(\text{Proposition \ref{subdistributor_inclusion}(3)})
\end{align*}

(iv)${}\Lra{}$(i): For all $\mu\in\PA$, since $\CM(I^{\nat},\bbB)$ is surjective and $\dphi\uphi\mu\in\CM\phi$, there exists $\mu'\in\CM\phi_{\bbA,\bbB'}\subseteq\PA'$ with
$\dphi\uphi\mu=\dphi\uphi\underline{\mu'}$. By Proposition \ref{subdistributor_inclusion}(3) one soon has
$$\phi\lda\mu=\phi\lda\underline{\mu'}=\phi_{\bbA',\bbB}\lda\mu',$$
and consequently $\bbA\setminus\bbA'$ is $\phi$-reducible.
\end{proof}

\begin{prop} \label{codomain_reducible_condition}
Let $\phi:\bbA\oto\bbB$ be a $\CQ$-distributor and $\bbB'\subseteq\bbB$. The following statements are equivalent:
\begin{itemize}
\item[\rm (i)] $\bbB\setminus\bbB'$ is $\phi$-reducible.
\item[\rm (ii)] $\lam\rda\phi=(\uphi\dphi\lam)_{\bbB'}\rda\phi_{\bbA,\bbB'}$ for all $\lam\in\PdB$.
\item[\rm (iii)] $\dphi\uphi=\dphiABp\uphiABp:\PA\to\PA$.
\item[\rm (iv)] $\CM(\bbA, J_{\nat}):\CM\phi\to\CM\phi_{\bbA,\bbB'}$ is the identity $\CQ$-functor on $\CM\phi=\CM\phi_{\bbA,\bbB'}$.
\end{itemize}
\end{prop}

\begin{proof}
(i)$\iff$(ii)$\iff$(iii) is the dual of the equivalences of (i), (ii), (iii) in Proposition \ref{domain_reducible_condition}.

(iii)${}\Lra{}$(iv): $\CM\phi=\CM\phi_{\bbA,\bbB'}$ is obvious and
$$\CM(\bbA, J_{\nat})\mu=\dphiABp\uphiABp\bbA^*\mu=\dphi\uphi(\mu\circ\bbA)=\dphi\uphi\mu=\mu$$
for all $\mu\in\CM\phi=\CM\phi_{\bbA,\bbB'}$.

(iv)${}\Lra{}$(iii): $\CM\phi$ and $\CM\phi_{\bbA,\bbB'}$ are, by definition, respectively the fixed points of the   $\CQ$-monads $\dphi\uphi:\PA\to\PA$ and $\dphiABp\uphiABp:\PA\to\PA$ on skeletal $\CQ$-categories; so, if $\CM\phi=\CM\phi_{\bbA,\bbB'}$, it follows from Proposition \ref{monad_reflective}(1) that, when restricting the codomain to the image, both $\dphi\uphi$ and $\dphiABp\uphiABp$ are left adjoint to the same inclusion $\CQ$-functor, thus they must be equal.
\end{proof}

Now we are ready to complete the proof of Theorem \ref{reduct reducible}:

\begin{proof}[Proof of Theorem \ref{reduct reducible}]
If $\bbA\setminus\bbA'$ and $\bbB\setminus\bbB'$ are both $\phi$-reducible, it is easy to check that $\bbB\setminus\bbB'$ is $\phi_{\bbA',\bbB}$-reducible by help of Proposition \ref{subdistributor_inclusion}(2). Then it follows from Propositions \ref{domain_reducible_condition} and \ref{codomain_reducible_condition} that both $\CM(I^{\nat},\bbB):\CM\phi_{\bbA',\bbB}\to\CM\phi$ and $\CM(\bbA', J_{\nat}):\CM\phi_{\bbA',\bbB}\to\CM\phi_{\bbA',\bbB'}$ are isomorphisms. Thus   $(-)_{\bbA'}:\CM\phi\to\CM\phi_{\bbA',\bbB}$, the inverse of $\CM(I^{\nat},\bbB)$, is an isomorphism by Lemma \ref{A_JB_Chu_connection_M}. Therefore
$$R_1=(\CM\phi\to^{(-)_{\bbA'}}\CM\phi_{\bbA',\bbB}\to^{\CM(\bbA', J_{\nat})}\CM\phi_{\bbA',\bbB'})$$
is an isomorphism, showing that $\phi_{\bbA',\bbB'}$ is a reduct of $\phi$.

Conversely, suppose that $\phi_{\bbA',\bbB'}$ is a reduct of $\phi$.  By definition, both $E_1$ and $E_2$ are isomorphisms, thus $\CM(I^{\nat},\bbB):\CM\phi_{\bbA',\bbB}\to\CM\phi$ and the inclusion $\CQ$-functor $\CM\phi_{\bbA,\bbB'}\ \to/^(->/\CM\phi$ are both surjective; in particular, the inclusion $\CQ$-functor $\CM\phi_{\bbA,\bbB'}\ \to/^(->/\CM\phi$ is the identity $\CQ$-functor on $\CM\phi_{\bbA,\bbB'}=\CM\phi$, and so is its left adjoint $\CM(\bbA, J_{\nat})$. Therefore, $\bbA\setminus\bbA'$ is $\phi$-reducible by Proposition \ref{domain_reducible_condition}, and $\bbB\setminus\bbB'$ is $\phi$-reducible by Proposition \ref{codomain_reducible_condition}.
\end{proof}

\begin{exmp}
Given $A,B\subseteq\bQ$, $(A,B,\leq)$ is a reduct of $(\bQ,\bQ,\leq)$ if and only if both $A$ and $B$ are dense in $\bQ$. To see this, note that $\bQ\setminus A$ being $\leq$-reducible exactly means for all $x\in\bQ\setminus A$, by Remark \ref{FCA_reducible}, there exists a subset $A'\subseteq A$ with $\ua x=\ua A'$, where $\ua A'$ denotes the set of upper bonds of $A'$ in $\bQ$; or equivalently, for all $x\in\bQ\setminus A$, $x=\displaystyle\bv A'$ for some $A'\subseteq A$; that is, $A$ is dense in $\bQ$. Since the characterization of the $\leq$-reducibility of $\bQ\setminus B$ may be obtained dually, the conclusion then follows from Theorem \ref{reduct reducible}.
\end{exmp}

\begin{exmp}
Let $X$ be a topological space and $\CF$ the set of closed sets of $X$. Define a relation $R\subseteq X\times\CF$ as
$$(x,F)\in R\iff x\in F,$$
then for any subset $\CF'\subseteq\CF$, $(X,\CF',R_{X,\CF'})$ is a reduct of $(X,\CF,R)$ if and only if $\CF'$ is a base for the closed sets of $X$. For this one notices
\begin{equation} \label{dRA_intersection}
\dR(\CA)=\{x\in X\mid\forall F\in\CA:\ x\in F\}=\bigcap\CA
\end{equation}
for all $\CA\subseteq\CF$, and consequently
\begin{align*}
\CF\setminus\CF'\ \text{is}\ R\text{-reducible}&\iff\forall F\in\CF\setminus\CF':\ \dR(\{F\})=\dR(\CA')\ \text{for some}\ \CA'\subseteq\CF'&(\text{Remark \ref{FCA_reducible}})\\
&\iff\forall F\in\CF\setminus\CF':\ F=\bigcap\CA'\ \text{for some}\ \CA'\subseteq\CF'&(\text{Equation (\ref{dRA_intersection})})\\
&\iff\CF'\ \text{is a base for the closed sets of}\ X.
\end{align*}
\end{exmp}

Another fact emerged from Propositions \ref{domain_reducible_condition} and \ref{codomain_reducible_condition} is that $R_1:\CM\phi\to\CM\phi_{\bbA',\bbB'}$ is a left adjoint $\CQ$-functor if $\bbA\setminus\bbA'$ is $\phi$-reducible, and so is $E_2:\CM\phi_{\bbA',\bbB'}\to\CM\phi$ if $\bbB\setminus\bbB'$ is $\phi$-reducible. The fullness of the functor $\CM:\ChuCon(\QDist)^{\op}\to\QCCat$ (Proposition \ref{M_functor}) then implies that they must be induced by some Chu connections. We spell this out in the following:

\begin{prop}
Let $\phi:\bbA\oto\bbB$ be a $\CQ$-distributor and $\bbA'\subseteq\bbA$, $\bbB'\subseteq\bbB$.
\begin{itemize}
\item[\rm (1)] If $\bbA\setminus\bbA'$ is $\phi$-reducible, then $(\phi\rda\phi_{\bbA',\bbB},\ J_{\nat}):\phi_{\bbA',\bbB'}\to\phi$ is a Chu connection, and
$$\CM(\phi\rda\phi_{\bbA',\bbB},\ J_{\nat})=R_1:\CM\phi\to\CM\phi_{\bbA',\bbB'}.$$
\item[\rm (2)] If $\bbB\setminus\bbB'$ is $\phi$-reducible, then $(I^{\nat},\ \phi_{\bbA,\bbB'}\lda\phi):\phi\to\phi_{\bbA',\bbB'}$ is a Chu connection, and
$$\CM(I^{\nat},\ \phi_{\bbA,\bbB'}\lda\phi)=E_2:\CM\phi_{\bbA',\bbB'}\to\CM\phi.$$
\end{itemize}
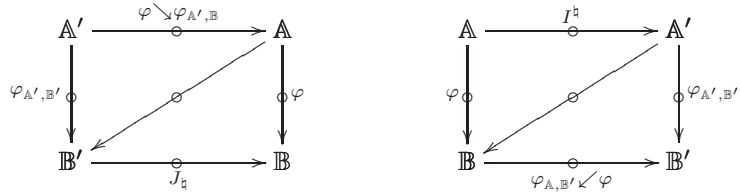
$$\bfig
\square<800,500>[\bbA'`\bbA`\bbB'`\bbB;\phi\rda\phi_{\bbA',\bbB}`\phi_{\bbA',\bbB'}`\phi`J_{\nat}]
\morphism(800,500)|b|<-800,-500>[\bbA`\bbB';]
\place(0,250)[\circ] \place(800,250)[\circ] \place(400,500)[\circ] \place(400,0)[\circ]\place(400,250)[\circ]
\square(1500,0)<800,500>[\bbA`\bbA'`\bbB`\bbB';I^{\nat}`\phi`\phi_{\bbA',\bbB'}`\phi_{\bbA,\bbB'}\lda\phi]
\morphism(2300,500)|b|<-800,-500>[\bbA'`\bbB;]
\place(1500,250)[\circ] \place(2300,250)[\circ] \place(1900,500)[\circ] \place(1900,0)[\circ]\place(1900,250)[\circ]
\efig$$
\end{prop}

\begin{proof}
(1) First, one always has
$$\phi_{\bbA,\bbB'}=J^{\nat}\circ\phi=J_{\nat}\rda\phi,$$
where the two equalities respectively follow from Proposition \ref{subdistributor_inclusion}(1) and Proposition \ref{adjoint_arrow_calculation}(1).

Second, in the case that $\bbA\setminus\bbA'$ is $\phi$-reducible,
\begin{align*}
\phi_{\bbA,\bbB'}&=J^{\nat}\circ(\phi\lda(\phi\rda\phi))
\\
&=J^{\nat}\circ(\phi_{\bbA',\bbB}\lda(\phi\rda\phi_{\bbA',\bbB}))&(\text{Proposition \ref{domain_reducible_condition}(iii)})\\
&=(J^{\nat}\circ\phi_{\bbA',\bbB})\lda(\phi\rda\phi_{\bbA',\bbB})&(\text{Proposition \ref{adjoint_arrow_calculation}(3)})\\
&=\phi_{\bbA',\bbB'}\lda(\phi\rda\phi_{\bbA',\bbB}).&(\text{Proposition \ref{subdistributor_inclusion}(1)})
\end{align*}
Thus $(\phi\rda\phi_{\bbA',\bbB},\ J_{\nat}):\phi_{\bbA',\bbB'}\to\phi$ is a Chu connection.

Finally, for its image under $\CM$, note that for all $\mu\in\CM\phi$,
\begin{align*}
R_1\mu&=(\phi_{\bbA',\bbB'}\lda\mu_{\bbA'})\rda\phi_{\bbA',\bbB'}\\
&=(\phi_{\bbA',\bbB'}\lda(\dphi\uphi\mu)_{\bbA'})\rda\phi_{\bbA',\bbB'}&(\mu\in\CM\phi)\\
&=(\phi_{\bbA,\bbB'}\lda\mu)\rda\phi_{\bbA',\bbB'}&(\text{Proposition \ref{domain_reducible_condition}(ii)})\\
&=((\phi_{\bbA',\bbB'}\lda(\phi\rda\phi_{\bbA',\bbB}))\lda\mu)\rda\phi_{\bbA',\bbB'}&(\text{the second step})\\
&=(\phi_{\bbA',\bbB'}\lda(\mu\circ(\phi\rda\phi_{\bbA',\bbB})))\rda\phi_{\bbA',\bbB'}\\
&=\dphiApBp\uphiApBp(\phi\rda\phi_{\bbA',\bbB})^*\mu\\
&=\CM(\phi\rda\phi_{\bbA',\bbB},\ J_{\nat})\mu,
\end{align*}
where the third equality holds because $\bbA\setminus\bbA'$ is $\phi_{\bbA,\bbB'}$-reducible by Proposition \ref{subdistributor_inclusion}(2).

(2) By applying (1) to the $\CQ^{\op}$-distributor $\phi^{\op}:\bbB^{\op}\oto\bbA^{\op}$ one obtains that
$$(\phi^{\op}\rda(\phi^{\op})_{\bbB'^{\op},\bbA^{\op}},\ (I^{\op})_{\nat}):(\phi^{\op})_{\bbB'^{\op},\bbA'^{\op}}\to\phi^{\op}$$
is a Chu connection if $\bbB^{\op}\setminus\bbB'^{\op}$ is $\phi^{\op}$-reducible. By duality (see the isomorphism (\ref{ChuCon_QDist_op})) this exactly means $(I^{\nat},\ \phi_{\bbA,\bbB'}\lda\phi):\phi\to\phi_{\bbA',\bbB'}$ is a Chu connection if $\bbB\setminus\bbB'$ is $\phi$-reducible, since it is easy to see $(I^{\nat})^{\op}=(I^{\op})_{\nat}$. For its image under $\CM$, note that for all $\mu'\in\CM\phi_{\bbA',\bbB'}$,
\begin{align*}
E_2\mu'&=\dphiABp\uphiABp\underline{\mu'}\\
&=\dphi\uphi\underline{\mu'}&(\text{Proposition \ref{codomain_reducible_condition}(iii)})\\
&=\dphi\uphi(\mu'\circ I^{\nat})&(\text{Formulas (\ref{underline_overline_phi_def})})\\
&=\dphi\uphi(I^{\nat})^*\mu'\\
&=\CM(I^{\nat},\ \phi_{\bbA,\bbB'}\lda\phi)\mu',
\end{align*}
completing the proof.
\end{proof}

Since $\ChuCon(\QDist)$ is a quantaloid,   four Chu connections between $\phi:\bbA\oto\bbB$ and $\phi_{\bbA',\bbB'}:\bbA'\oto\bbB'$ can be constructed from the Chu connections $(\bbA, J_{\nat}): \phi_{\bbA,\bbB'}\to\phi$ and $(I^{\nat},\bbB):\phi\to\phi_{\bbA',\bbB}$:
$$\bfig
\Atriangle|lra|/->`->`@{->}@<3pt>/<700,500>[\phi_{\bbA,\bbB'}`\phi_{\bbA',\bbB'}`\phi;
(I^{\nat},\bbB')`(\bbA, J_{\nat})`(\bbA, J_{\nat})\lda(I^{\nat},\bbB')]
\morphism(1400,0)|b|/@{->}@<3pt>/<-1400,0>[\phi`\phi_{\bbA',\bbB'};(I^{\nat},\bbB')\lda(\bbA, J_{\nat})]
\Vtriangle(2000,0)|blr|/@{->}@<-3pt>`->`->/<700,500>[\phi_{\bbA',\bbB'}`\phi`\phi_{\bbA',\bbB};
(I^{\nat},\bbB)\rda(\bbA', J_{\nat})`(\bbA', J_{\nat})`(I^{\nat},\bbB)]
\morphism(3400,500)|a|/@{->}@<-3pt>/<-1400,0>[\phi`\phi_{\bbA',\bbB'};(\bbA', J_{\nat})\rda(I^{\nat},\bbB)]
\efig$$
It is natural to ask whether reducts of $\CQ$-distributors are related to these Chu connections. However, due to the difficulty of calculating implications in $\ChuCon(\QDist)$, we failed to describe their images under $\CM:\ChuCon(\QDist)^{\op}\to\QCCat$. So, we end this paper with

\begin{ques}
Is it possible to characterize the reducts of a $\CQ$-distributor through the above implications of Chu connections?
\end{ques}

\section*{Acknowledgements}

The first author acknowledges the support of Natural Sciences and Engineering Research Council of Canada (Discovery Grant 501260 held by Professor Walter Tholen). The second and the third authors acknowledge the support of National Natural Science Foundation of China (11371265). We thank the anonymous referee for several helpful remarks.





\end{document}